\documentclass[12pt]{amsart}

\usepackage{amsfonts}
\usepackage{amsmath}
\usepackage{amscd}
\usepackage{amssymb}
\usepackage{amsthm}
\usepackage{bbm}
\usepackage{bm}
\usepackage{enumerate}
\usepackage{hyperref}
\usepackage{latexsym}
\usepackage{mathabx}
\usepackage{mathdots}
\usepackage{mathtools}
\usepackage[left=3cm,top=2cm,right=3cm,bottom = 2cm]{geometry}

\allowdisplaybreaks

\theoremstyle{definition}

\newtheorem{remark}{Remark}[section]
\theoremstyle{plain} 
\newtheorem{theorem}{Theorem}[section]
\newtheorem{lemma}{Lemma}[section]
\newtheorem{corollary}{Corollary}[section]
\newtheorem{proposition}{Proposition}[section]

\DeclareMathOperator{\diag}{diag}
\DeclareMathOperator{\id}{id}
\DeclareMathOperator{\Res}{Res}
\DeclareMathOperator{\SL}{SL}

\newcommand{\Atilde}{\widetilde{A}}
\newcommand{\Ltilde}{\widetilde{L}}

\newcommand{\Bound}{B^{1}(\Gamma,\HON)}
\newcommand{\Boundtwo}{B^{1}(\Gamma,\HOtwo)}  
\newcommand{\bfm}{\mathbf{m}}
\newcommand{\bfy}{\mathbf{y}}
\newcommand{\bfz}{\mathbf{z}}
\newcommand{\C}{\mathbb{C}}
\newcommand{\calA}{\mathcal{A}}
\newcommand{\calC}{\mathcal{C}}
\newcommand{\calF}{\mathcal{F}}

\newcommand{\calL}{\mathcal{L}}
\newcommand{\calM}{\mathcal{M}}
\newcommand{\calS}{\mathcal{S}}

\newcommand{\Cg}{\mathfrak{C}_{g}}
\newcommand{\Chat}{\widehat{\C}}
\newcommand{\COH}{H^{1}(\Gamma,\HON)}
\newcommand{\COHtwo}{H^{1}(\Gamma,\HOtwo)}
\newcommand{\Con}[1]{\mathcal{C}_{#1}}

\newcommand{\D}{\mathcal{D}}
\newcommand{\del}{\partial}
\newcommand{\delCg}{\nabla}
\newcommand{\nabmy}[2]{\nabla^{(#1)}
	_{#2}}
\newcommand{\delCgtilde}{\widetilde{\nabla}}

\newcommand{\delMg}{\nabla_{\hspace{-1 mm}\Mg}}

\newcommand{\dy}{\partial_{y}}
\newcommand{\dyb}{\partial_{\ybar}}

\newcommand{\Ei}{Z^{1}(\Gamma,\HON)}
\newcommand{\Eitwo}{Z^{1}(\Gamma,\HOtwo)}
\newcommand{\FN}{\calF_{N}}
\newcommand{\half}{\frac{1}{2}}
\newcommand{\Hg}{\mathcal H^{(g)}}
\newcommand{\HgN}{\Hg_{N}}
\newcommand{\Hgtwo}{\Hg_{2}}
\newcommand{\HH}{\mathbb{H}}
\newcommand{\Hm}{H_{\bfm}}
\newcommand{\HON}{{\mathcal H}^{(0)}_{1-N}}
\newcommand{\HOtwo}{{\mathcal H}^{(0)}_{-1}}

\newcommand{\I}{\mathcal{I}}
\newcommand{\im}{\textup{i}}

\newcommand{\Ip}{\mathcal{I}_{+}}

\newcommand{\Mg}{{\frak{M}}_{g}}

\newcommand{\Omo}{\Omega_{0}(\Gamma)}
\newcommand{\pmet}{\lambda} 
\newcommand{\PM}{\Lambda} 
\newcommand{\Poly}{\mathfrak {P}}

\newcommand{\Psitilde}{\widetilde{\Psi}}

\newcommand{\Rtilde}{\widetilde{R}}
\newcommand{\R}{\mathbb{R}}

\newcommand{\Schg}{\mathfrak{S}_{g}}

\newcommand{\secref}[1]{\S\ref{#1}}
\newcommand{\Sg}{\mathcal{S}^{(g)}}

\newcommand{\slLie}{\mathfrak{sl}}

\newcommand{\tpi}{2\pi \im} 
\newcommand{\Tr}{\textup{tr}\,}
\newcommand{\Utilde}{\widetilde{U}}
\newcommand{\vol}{\mathbf{v}}

\newcommand{\ybar}{\overline{y}}
\newcommand{\ygam}{y_{\gamma}}
\newcommand{\Ytilde}{\widetilde{Y}}
 
\newcommand{\Z}{\mathbb{Z}}
\newcommand{\zbar}{\overline{z}} 
\newcommand{\zgam}{z_{\gamma}}

\title{Some Properties and Applications of Bers Quasiforms on Riemann Surfaces}

\begin{document}
	\author{Michael P. Tuite}
	\address{School of Mathematical and Statistical Sciences\\ 
University of Galway, Galway H91 TK33, Ireland}
	\email{michael.tuite@universityofgalway.ie}
	\author[Michael Welby]{Michael Welby$^{\dagger}$}
	\thanks{$^{\dagger}$Supported by an Irish Research Council Government of Ireland Postgraduate Scholarship.}
\email{michael.welby@universityofgalway.ie}	
	\maketitle
	\begin{abstract}	
We describe some properties of the Bers quasiform
on a compact Riemann surface in the Schottky sewing scheme. Our main results are: (i) the expansion of meromorphic differential forms in terms of holomorphic forms and derivatives of the Bers quasiform, (ii) power series expansions in the Schottky sewing parameters of the Bers quasiform and holomorphic forms, (iii) a novel differential operator which acts on meromorphic forms in several variables which we apply in deriving differential equations for classical objects such as the bidifferential of the second kind, the projective connection, holomorphic 1-forms and the prime form. 
	\end{abstract}
	\section{Introduction} 	
In this paper we describe some novel properties for the Bers quasiform
on a compact genus $g$ Riemann surface $\Sg$ in the Schottky sewing scheme e.g. \cite{Fo,Bo}. The Bers quasiform, which we denote by $\Psi_{N}(x,y)$ for integer $N\ge 1$, was defined by Bers as a certain Poincar\'e series
and was introduced as a means for constructing Bers potentials for holomorphic weight $N$ differential forms  on $\Sg$ \cite{Be1,Be2}. 
Here we develop some general properties of the Bers quasiform which have important intrinsic meaning beyond their use for constructing potentials.
Our primary motivation is that the Bers quasiform plays a pivotal role in the description of vertex operator algebras and conformal field theories on a genus $g$ Riemann surface  \cite{TW}. Although this paper is designed to provide the necessary background for that application, we believe that our results should also be of wider interest to researchers in Riemann surface theory and its applications. 


Section 2 begins with an overview of the Schottky uniformization of a genus $g$ marked Riemann surface $\Sg$. We review the theory of Eichler cocycles,  Bers potentials and the Bers map for genus $g\ge 2$. The construction of a Bers potential for each holomorphic weight $N$ differential utilizes the  Poincar\'e series $\Psi_{N}(x,y)$  introduced by Bers \cite{Be1,Be2}. $\Psi_{N}(x,y)$ is meromorphic for $x,y\in\Sg$ with a simple pole at $x=y$,  transforms like a weight $N$ differential in $x$ but as a quasiperiodic differential of weight $1-N$  in $y$. $\Psi_{N}(x,y)$ can be thought of a generalization to $N>1$ of the classical differential of the third kind $\omega_{y-0}(x)$ e.g. \cite{Mu,Fa}.  We also discuss the non-uniqueness of the Bers quasiform and its relationship to the unique Greens function $G_{N}(x,y)$ for the anti-holomorphic part of the Poincar\'e metric compatible  connection \cite{Ma,McI,McIT}. 

In Section 3 we describe a novel expansion of any meromorphic $N$-form $H_{N}$ in terms of holomorphic $N$-forms and $y$ derivatives of $\Psi_{N}(x,y)$ with coefficients given by certain residues of $H_{N}$. This is a natural extension of genus zero and one expansions and precisely matches formal Zhu reduction formulas for vertex operator algebra correlation functions at genus $g\ge 2$ \cite{TW}.

Section 4 contains power series expansions of $\Psi_{N}$ and a canonical spanning set of holomorphic $N$ forms in terms of Schottky sewing parameters. This is a natural extension of expansion formulas for the classical bidifferential of the second kind and holomorphic 1-forms of \cite{Fa,Y,MT}. These expressions canonically arise in the context of vertex operator algebra theory \cite{TW}. We also describe a Selberg zeta function infinite product formula related to the determinant of the Laplacian acting on $N$-differentials \cite{McIT}. 

In Section 5 we specialize to the geometrically important case of the $N=2$  Bers quasiform $\Psi_{2}(x,y)$. In particular, we describe a differential operator $\nabla(x)$ which maps differentiable functions of the Schottky parameters  to the space of holomorphic 2-forms in $x$ (and naturally leads to a similar operator acting on differentiable functions on moduli space). We generalize this to a differential operator $\nabmy{\bfm}{\bfy}(x)$ that utilizes  $\Psi_{2}(x,y)$ and which maps any meromorphic form  $H_{\bfm}(\bfy)=h_{\bfm}(\bfy)dy_{1}^{m_{1}} \ldots dy_{n}^{m_{n}}$ of weight $(\bfm):=(m_{1},\ldots ,m_{n})$ in $\bfy:=y_{1},\ldots,y_{n}$ to a meromorphic form of weight $(2,\bfm)$ in $x,\bfy$. 
We illustrate this operator by describing their action on a number of classical objects such as the bidifferential of the second kind, the projective connection,  holomorphic 1-forms and the prime form. These results are obtained by alternative vertex operator algebra methods in \cite{TW}. Some of these results were also partially anticipated in ref.~\cite{O}. 
We conclude  by showing that the composition of two such differential operators (with suitable labeling) is commutative.

\noindent \textbf{Acknowledgements.} 
We wish to thank Michael Flattery and Tom Gilroy for helpful comments and suggestions.

\section{The Bers Quasiform} \label{sec:PsiN}
\subsection{Notational Conventions}
Define indexing sets for integers $g,N\ge 1$ as follows:
\begin{align}
\label{eq:ILconv}
	\I:=\{-1,\ldots ,- g,1,\ldots ,g\},\quad
	\Ip:=\{1,\ldots ,g\},\quad 
	\calL:=\{0,1,\ldots,2N-2\}.
\end{align}
For functions $f(x)$ and $g(x,y)$ and integers $i,j\ge 0$ we define
\begin{align}
\label{eq:delconv}
f^{(i)}(x):=\del_{x}^{(i)} f(x):= \frac{1}{i!}\frac{\del^{i}}{\del x^{i}}f(x),
\quad
g^{(i,j)}(x,y):=\del_{x}^{(i)} \del_{y}^{(j)}  g(x,y).
\end{align} 
Let $\Poly_{n}$ denote the space of polynomials with complex coefficients of degree $\le n$. 
\subsection{The Schottky uniformization of a Riemann surface}
\label{subsec:Schottky}
We briefly review the  construction of a genus $g$ Riemann surface $\Sg$ using the Schottky uniformization where we sew $g$ handles to the Riemann sphere $\calS^{(0)}\cong\Chat:=\C\cup\{\infty\}$ e.g. \cite{Fo,Bo}. Every Riemann surface can be Schottky uniformized \cite{Be3,Be4}.
Let $\{ \Delta_{a}\}$, for $a\in\I$, denote $2g$ non-intersecting simply connected regions in $\C$ with smooth boundary curves $ \{\Con{a}\}$. 
Identify $z'\in \Con{-a}$ with $z\in  \Con{a}$ via the Schottky sewing cross ratio relation
\begin{align}
\label{eq:Schottkysew}
\frac{z'-W_{-a}}{z'-W_{a}}\,\frac{z-W_{a}}{z-W_{-a}}=q_{a},\quad a\in\Ip,
\end{align}
for complex $q_{a}$ with $0<|q_{a}|<1$ and some $W_{\pm a}\in\Chat$. Thus  $z'=\gamma_{a}z$  for $a\in\Ip$ for  M\"obius transformation generated by $\gamma_{a}\in \SL_{2}(\C)$  where
\begin{align}\label{eq:gammaa}
\gamma_{a}:=\sigma_{a}^{-1}
\begin{pmatrix}
q_{a}^{1/2} &0\\
0 &q_{a}^{-1/2}
\end{pmatrix}
\sigma_{a},\quad 
\sigma_{a}=(W_{-a}-W_{a})^{-1/2}\begin{pmatrix}
1 & -W_{-a}\\
1 & -W_{a}
\end{pmatrix}.
\end{align}
$\gamma_{a}$ is loxodromic with attracting fixed point $W_{-a}\in\Delta_{-a}$ and repelling fixed point $W_{a}\in\Delta_{a}$.
 $\gamma_{a}$ maps the interior/exterior of $\Delta_{a}$ 
 to the exterior/interior of $\Delta_{-a}$.

The marked Schottky group $\Gamma\subset \SL_{2}(\C)$  is the free discrete group of M\"obius transformations generated by $\gamma_{a}$. 
Let $\Lambda(\Gamma)$ denote the  limit set for $\Gamma$.
Let $\Omo=\Chat-\Lambda(\Gamma)$ denote the region of
discontinuity. 
Then  $\Sg\simeq \Omo/\Gamma$, a Riemann surface of genus $g$. 
We let $\D = \Chat/ \cup_{a\in\I}\Delta_{a}$ denote the standard connected fundamental region with oriented boundary curves $\Con{a}$. 
We further identify the standard homology cycle $\alpha_{a}$  with $ \Con{-a}$ and the cycle $\beta_{a}$ with a path connecting  $z\in  \Con{a}$ to $z'=\gamma_{a}z\in  \Con{-a}$. 

We define the space of Schottky parameters  $\Cg \subset \C^{3g}$ by
\begin{align}
	\label{eq:Cfrakg}
	\Cg:=\left\{ (W_{1},W_{-1},q_{1},\ldots,W_{g},W_{-g},q_{g}): \Delta_{a}\cap\Delta_{b}=\O\; \forall \;a\neq b\right\}.
\end{align}
The cross ratio \eqref{eq:Schottkysew} is M\"obius  invariant with $(z,z',W_{a},q_{a})\mapsto(\gamma z,\gamma z',\gamma W_{a},q_{a})$ for $\gamma\in\SL_{2}(\C)$.  We define Schottky space by $\Schg=\Cg/\SL_{2}(\C)$. $\Schg$ is a covering space for the $3g-3$ dimensional moduli space $\mathcal{M}_{g}$ of genus $g$ Riemann surfaces e.g. \cite{Be3}.

It is convenient to introduce alternative parameters for $\Cg$ as follows. Define $\gamma_{-a}=\gamma_{a}^{-1} $ so that $\gamma_{a} \Con{a}= -\Con{-a}$ for all $a\in\I$. 
Let $
w_{a}:=\gamma_{-a}.\infty\in\Delta_{a}$. Hence $
\frac{w_{a}-W_{a}}{w_{a}-W_{-a}}=q_{a}
$
which implies
\begin{align}
\label{eq:wadef}
w_{a} ={\frac {W_{{a}}-q_{{a}}W_{{-a}}}{1-q_{{a}}}},\quad a\in\I.
\end{align} 
We therefore have
\begin{align}
\label{eq:gamaz}
\gamma_{a}z = w_{-a}+\frac{\rho_{a}}{z-w_{a}},
\end{align}  
where $\rho_{a}=\rho_{-a}$ is 
determined from the condition $\gamma_{a}W_{a}=W_{a}$ to be
\begin{align}
\label{eq:rhoadef}
\rho_{a}=-{\frac {q_{{a}} \left( W_{{a}}-W_{{-a}} \right) ^{2}}{\left(1- q_{{a}} \right) ^{2}}}.
\end{align} 
Hence \eqref{eq:Schottkysew} can be written in the more convenient form: 
\begin{align}
\label{eq:Cansew}
(z'-w_{-a})(z-w_{a})=\rho_{a}.
\end{align}
The M\"obius  action of $\gamma =\left(\begin{smallmatrix}A&B\\C&D\end{smallmatrix}\right)
$ on the Schottky parameters $(w_{a},\rho_{a})$  is given by
\begin{align}
	\label{eq:Mobwrhoa}
	\gamma:(w_{a},\rho_{a})\mapsto & 
	\left(	\frac { \left( Aw_{a}+B \right)  \left( Cw_{-a}+D \right) -\rho_{a}
		\,AC}{ \left( Cw_{a}+D \right)  \left( Cw_{-a}+D \right) -\rho_{a}\,{
			C}^{2}},
	{\frac {\rho_{a}}{ \left(  \left( Cw_{a}+D \right)  \left( Cw_{-a}+D
			\right) -\rho_{a}\,{C}^{2} \right) ^{2}}}\right).
\end{align}

\subsection{Holomorphic  differentials ${\mathcal H}^{(g)}_{N}$}
Let ${\mathcal A} _{m,n}$ be the space of smooth differentials of weight $(m,n)$,  for $m,n\in\Z$,
of the form $\Phi(z)=\phi(z) dz^{m}d\zbar^{n}$ for local coordinate $z$ on $\Sg$ e.g.
the Poincar\'e  metric\footnote{$\PM$ is induced from the Poincar\'e metric $y^{-2}d\zeta d\overline{\zeta}$ on  $\HH=\{\zeta=x+ \im\,y|x,y\in\R,\, y>0\}$ by uniformization $\Sg$ as a quotient of $\HH$ by an appropriate Fuchsian subgroup of $ \SL_{2}(\R)$.} 
$\PM(z)=\pmet(z)dz d\zbar \in{\mathcal A} _{1,1}$ for real positive $\pmet(z)$. $\PM$ determines the  positive definite Petersson inner product  for $\Phi,\Theta \in {\mathcal A} _{m,n}$ defined by
\begin{align}
\label{eq:Pprod}
\langle \Phi,\Theta \rangle   
:= \iint_{\Sg}\Phi\,
\overline{\Theta}\,\PM^{-m-n}\vol, 
\end{align} 
for real volume form $
\vol(z) :=\frac{\im}{2}\pmet(z) dz\wedge d\zbar$.

Let $\mathfrak{H}^{(g)}_{m,n}$ denote the $L^2$-closure of ${\mathcal A} _{m,n}$ with respect to the Petersson product and let $\mathfrak{H}^{(g)}_{n}=\mathfrak{H}^{(g)}_{n,0}$.
Lastly, let $\HgN\subset \mathfrak{H}^{(g)}_{N}$ denote the space of genus $g$ \emph{holomorphic} $N$-differentials of integral weight $N$.
The Riemann-Roch theorem (e.g. \cite{FK,Bo}) determines $d_{N}=\dim{\HgN}$ as follows:  
\begin{center}
	\begin{tabular}{|c|c|c|}
		\hline
		$g$  & $N\in\Z$ & $ d_{N}$ \\
		\hline
		$g=0$ & $N\le 0$ & $1-2N$\\
	 & $N> 0$ & $0$ \\
	 \hline
	$g=1$ & all $N$ & $1$ \\
	 \hline
	$g\ge 2$ & $N< 0$ & $0$ \\
	 & $N=0$ & $1$ \\
	 & $N=1$ & $g$  \\
	 & $N\ge 2$ & $(g-1)(2N-1)$ \\
	 \hline
	\end{tabular}
~\vskip 0.2 cm~
\end{center} 
In the Schottky uniformization, $\Phi(z)=\phi(z) dz^{m}d\zbar^{n}\in {\mathcal A} _{m,n}$
for $z\in\Omo$ satisfies
\begin{align*}
\Phi|_{\gamma}=\Phi,
\end{align*}
for all $\gamma\in \Gamma$ where $\Phi|_{\gamma}(z):=\phi(\gamma z) d(\gamma z)^md(\overline{\gamma z})^n$. The Petersson product is expressed as an integral over 
the Schottky fundamental region $\D$ e.g. \cite{Be1, McIT}.

\subsection{Eichler cocycles, Bers potentials and the Bers quasiform}
We review the relationship between $\HON$ and $\HgN$ in the Schottky scheme for all $N\ge 2$  and $g\ge 2$ as developed by Bers \cite{Be1,Be2,K}.
$\HON$ consists of elements
 $P(z)=p(z)dz^{1-N}$ for  $z\in \Chat$ and $p\in \Poly_{2N-2}$.
There is a natural M\"obius action on $\HON$ given by 
\begin{align}
\label{eq:MobPi}
P|_\gamma (z):=P (\gamma z),
\end{align}
for $\gamma\in\SL_{2}(\C)$ with $P|_{\gamma\lambda}=P|_{\gamma}|_{\lambda}$ for all $\gamma,\lambda\in\SL_{2}(\C)$.

Let $\Gamma$ be a Schottky group for a Riemann surface of genus $g\ge 2$.
Let $\Ei$ denote the vector space of Eichler 1-cocycles\footnote{We define  a 1-cocycle  as a mapping to $\HON$ rather than to $\Poly_{2N-2}$ as in refs.~\cite{Be1,K,G}.} for $\Gamma$ given by mappings of the form $\Xi:\Gamma\rightarrow \HON$ such that for all $\gamma,\lambda\in\Gamma$
\begin{align}
\label{eq:cocycle}
\Xi[\gamma\lambda]=\Xi[\gamma]|_{\lambda}+\Xi[\lambda],
\end{align} 
for M\"obius action \eqref{eq:MobPi}. We note that \eqref{eq:cocycle} implies $\Xi[\id]=0$ so that 
$\Xi[\gamma]=-\Xi[\gamma^{-1}]|_{\gamma}$.
Let $\Bound\subset \Ei$ denote the  space of coboundaries  
$\Xi_{P}:\Gamma\rightarrow \HON$ for $P\in \HON$ defined by
\begin{align}
\label{eq:cobound}
\Xi_{P}[\gamma]:=P|_{\gamma}-P.
\end{align}
$\Xi_{P}$ is a 1-cocycle since $P|_{\gamma}|_{\lambda}-P= \left(P|_{\gamma}-P\right)|_{\lambda}+P|_{\lambda}-P$. It is easy to show that \cite{Be1}  
\begin{lemma}\label{lem:dimcobound}
$\Bound\simeq \HON$ as vector spaces.
\end{lemma}

Let $\COH:=\Ei/\Bound$ be the cohomology space of Eichler cocycles modulo  coboundaries.
\begin{lemma}{\cite{Be1}}\label{lem:dimH}
$\dim \COH =(g-1)(2N-1)$. 
\end{lemma} 
\begin{proof}
Since $\Gamma$ is freely generated by $\{\gamma_{a}\}$ for $a\in\Ip$, a cocycle $\Xi$ is determined by its evaluation on $\gamma_{a}$. Thus 
\[
\dim\Ei=g\dim\HON=g(2N-1). 
\]
The result follows  on applying Lemma~\ref{lem:dimcobound}.
\end{proof}
Note that $\dim\COH=\dim\HgN$ for $N\ge 2$ and $g\ge 2$. The Bers map $\beta$ discussed below describes  a bijection between these spaces. 
\medskip

$F(y)=f(y)dy^{1-N}$ for $f(y)$ continuous for $y\in\Omo$  is called a Bers potential for a holomorphic $N$-form $ \Phi =\phi (y)dy^{N}\in \HgN$ provided $f(y)$ satisfies\footnote{We include a factor of $\frac{1}{\pi}$ in comparison to \cite{Be1, McIT} for later convenience.}
\begin{align}
\label{eq:dF}
&\frac{1}{\pi}\dyb f=\overline{\phi (y)}\,\pmet(y)^{1-N},
\\
&\lim_{y\rightarrow 0}\left\vert y ^{2N-2}f\left(y^{-1}\right)\right\vert <\infty, 
\label{eq:Flim}
\end{align} 
for Poincar\'e metric $\PM(z)=\pmet(z)dz d\zbar$.  We note  $F(y)+P(y)$ is a potential for $\Phi$ any $P\in\HON$. 
Define the generalized Beltrami differential associated with $\Phi\in\HgN$ by 
\begin{align}
	\label{eq:Beldef}
B_{\Phi}:=\overline{\Phi}\,\PM^{-N}\vol \in\calA_{1-N,1},
\end{align} 
for volume form $\vol$.
$B_{\Phi}$ is an element of the dual space of $\HgN$ with respect to the canonical pairing
\begin{align}
	\label{eq:GenBel}
	\iint_{\D}\Theta B_{\Phi}=\langle \Theta,\Phi\rangle,\quad \Theta \in\HgN.
\end{align}
Then \eqref{eq:dF} can be written in the following coordinate-free way 
\begin{align}
\label{eq:dPhiF}
\frac{1}{\tpi}d \left(\Theta  F\right)=\Theta \,B_{\Phi},
\end{align} 
for all $\Theta \in\HgN$ with exterior derivative $d\left(h(y)dy\right)=-\dyb h \,dy\wedge d\ybar$.
\eqref{eq:Flim} ensures that $F(y)$ is defined at the point at infinity.
For $N=2$, $B_{\Phi}$ is the classical Beltrami differential associated with quasi-conformal maps where \eqref{eq:dF} is related to the Beltrami equation. This is explored further in Section~\ref{sec:Moduli}.

 We now define the Bers quasiform $\Psi_{N}(x,y)$ which is used to establish  the existence of a Bers potential for each $\Phi \in \HgN$ for $N\ge 2$ and $g\ge 2$  \cite{Be1,Be2,Be3,Be4,K,G}. 
The Bers quasiform is defined by the following Poincar\'e series 
 \begin{align}
 	\label{eq:PsiNdef}
 	\Psi_{N}(x,y):=\sum_{\gamma\in\Gamma} 
\Pi_{N}(\gamma x,y),\quad x,y\in \Omo,
 \end{align}
where
\begin{align}
	\label{eq:PiNdef}
	\Pi_{N}(x,y)&:=\pi_{N}(x,y)dx^{N}dy^{1-N},\quad 
	\pi_{N}(x,y):=\frac{1}{x-y}\prod_{j\in\calL}\frac{y-A_{j}}{x-A_{j}},
\end{align}
and where $\{A_{j}\}$ for $j\in \calL$, the indexing set  \eqref{eq:ILconv}, are distinct points in the limit set $\Lambda(\Gamma)$. 
 $\Psi_{N}(x,y)$ is meromorphic  in $x,y\in \Omo$ with a simple pole  of residue one at $y=\gamma x$ for all $\gamma\in\Gamma$  \cite{Be1}. 
 
 $\Psi_{N}(x,y)$ is a bidifferential $(N,1-N)$-quasiform with respect to the Schottky group as follows.
 By construction, $\Psi_{N}(x,y)$ is an $N$-differential in $x$ so that
 \begin{align}
 	\Psi_{N}(\gamma x,y) &= \Psi_{N}(x,y),\quad \gamma\in\Gamma,
 	\label{eq:GEM3rsx}
 \end{align}
 whereas $\Psi_{N}(x,y)$ is a quasiperiodic $1-N$ differential in $y$ with  
 \begin{align}\label{eq:PsiNdefgamy}
 	\Psi_{N}( x, \gamma y) =
 	\Psi_{N}(x,y)+\chi_{N}[\gamma](x,y),\quad \gamma\in\Gamma,
 \end{align}
 where $\chi_{N}[\gamma](x,y)$ is holomorphic for $x,y\in\Omo$ with an explicit form shown below in \eqref{eq:chiPhidual} and \eqref{eq:chiTheta}.
 
 We may also define a symmetric bidifferential form  $\omega_{N}(x,y)$  of weight $(N,N)$  by
 \begin{align}
 	\label{eq:Lambda}
 	\omega_{N}(x,y):=&\Psi_{N}^{(0,2N-1)}(x,y)dy^{2N-1}
 	=\sum_{\gamma \in\Gamma}M_{N}(\gamma x, y),
 \end{align}
 recalling convention \eqref{eq:delconv} and where $M_{N}(x,y):=\frac{dx^{N}dy^{N}}{(x-y)^{2N}}$. $\omega_{N}(x,y)$ generalizes the classical bidifferential  $\omega(x,y)=\omega_{1}(x,y)$  of weight $(1,1)$ e.g. \cite{Mu,Fa}.

 For a given $N$-differential $ \Phi  \in \HgN$ and $y\in\Omo$, then 
 \begin{align}\label{eq:FPsi}
 	F(y)=-\iint_{\D}\Psi_{N}(\cdot,y)\,B_{\Phi}
 	=-\langle \Psi_{N}(\cdot,y), \Phi \rangle.
 \end{align}
satisfies \eqref{eq:dF} and \eqref{eq:Flim} \cite{Be1}. Thus we find 
 \begin{proposition}{\cite{Be1}}
 	There exists a Bers potential  for each $ \Phi  \in \HgN$.
\end{proposition}
 Let $\calF_{N}$ denote the vector space of  Bers potentials for $\HgN$. 
\begin{lemma}{\cite{Be1}}\label{lem:ZeroPot}
$F\in\FN$ is a  Bers potential for $\Phi =0$ iff  $F\in\HON$.
\end{lemma}
For $F\in\FN$ define for $\gamma\in\Gamma$ 
\begin{align}
\Xi_{F}[\gamma]:=F|_{\gamma}-F,
\label{eq:xiF}
\end{align}
where $F|_\gamma (y):=F(\gamma y)$.
We then find
\begin{lemma}\label{lem:xiF_cocycle}
$\Xi_{F}$ is an Eichler 1-cocycle for each  $F\in\FN$. 
\end{lemma}
  We note the following useful identity \cite{Be1,McIT}  
\begin{proposition}\label{prop:ThetaPhi}
	Let $\Theta,\Phi\in\HgN$. Then 
	\begin{align}
		\langle \Theta,\Phi \rangle =  \frac{1}{\tpi}\sum_{a\in\Ip}\oint_{\Con{a}}\Theta\,\Xi_{F}[\gamma_{a}],
		\label{eq:ThetaPhi}
	\end{align}
	where $F$ is any Bers potential for $\Phi$ with cocycle $\Xi_{F}$.
\end{proposition}
\begin{proof}
	Using  \eqref{eq:dPhiF}  we find using Stokes' theorem on $\mathcal D$  with boundary curves $\mathcal{C}_a$ that
	\begin{align*}\notag
	\tpi	\langle \Theta,\Phi \rangle=&
	 \iint_{\D}d\left( \Theta \,
		F \right)
		=
		- \sum_{a\in\I}\oint_{\Con{a}}\Theta F
		\\ 
		=&
		- \sum_{a\in\Ip}\oint_{\Con{a}}\Theta \,\left(F-F|_{\gamma_{a}}\right)
		=
		 \sum_{a\in\Ip}\oint_{\Con{a}}\Theta \,\Xi_{F}[\gamma_{a}],
	\end{align*}
since $\calC_{-a}=-\gamma_{a} \calC_{a}$.
\end{proof}
\begin{remark}\label{rem:Xirep}
	$\sum_{a\in\Ip}\oint_{\Con{a}}\Theta\,\Xi_{P}[\gamma_{a}]=0$ for any coboundary cocycle $\Xi_{P}$.
\end{remark}
\begin{corollary}\label{cor:PhiCont}
	Let $\{\Phi _{s}\}_{s=1}^{d_{N}}$  be a $\HgN$-basis with  Petersson dual basis $\{\Phi ^{\vee}_{r}\}_{r=1}^{d_{N}}$.
	For  $ r,s=1,\ldots,d_{N}$ we have\footnote{We note that there  appears to be a sign error  in  (4.1) of \cite{McIT}.}
	\begin{align}
		\frac{1}{\tpi}\sum_{a\in\Ip}\oint_{\Con{a}}\Phi ^{\vee}_{r}\,\Xi_{s}[\gamma_{a}]=\delta_{rs},
		\label{eq:intPhiXi}
	\end{align}
	where $\Xi_{s}$ is the cocycle for any Bers potential $F_{r}$ for $\Phi_{s}$.
\end{corollary}
The following commutative diagram summarizes the various maps introduced above: 
\begin{align}
	\label{eq:maps}
	\begin{CD}
		\calF_{N}  @>\alpha>> \Ei
		\\
		@V\epsilon VV @V\delta VV  
		\\
		\HgN  @>\beta >>  \COH
	\end{CD}
\end{align}
where $\epsilon $ is the complex anti-linear map determined by  \eqref{eq:dF},  $\alpha$ is the linear map determined by \eqref{eq:xiF} and $\delta$ is  the coboundary quotient map with $\Bound=\ker \delta $. Then 
Lemmas~\ref{lem:dimcobound} and \ref{lem:ZeroPot} are equivalent to
\begin{align}\label{eq:kernels}
	\ker \epsilon  =\HON,\quad \ker \delta=\alpha (\ker \epsilon  ).
\end{align}
The complex antilinear mapping $\beta $ is known as the Bers map.  
We have the following fundamental result  
\begin{proposition}{\cite{Be1}}\label{prop:BersMap}
The maps $\alpha$ and $\beta$ are  bijective.
\end{proposition} 
\begin{proof}
We first show that $\alpha$  is bijective. From Lemma~\ref{lem:xiF_cocycle} we have
$\dim\calF_{N}=\dim\Ei$  so that $\alpha$ is surjective. Let $F\in\calF_{N}$ be a potential for $\Phi\in\HgN$ such that $\Xi_{F}=\alpha(F)=0$. 
Proposition~\ref{prop:ThetaPhi} implies that 
\begin{align*}
\langle \Phi,\Phi\rangle=\frac{1}{\tpi}\sum_{a\in\Ip}\oint_{\Con{a}}\Phi \,\Xi_{F} =0.
\end{align*}
Since $\langle \, , \,\rangle$ is positive definite, it follows that  $\Phi=0$ and hence $F$ is holomorphic by \eqref{eq:dPhiF}. 
But \eqref{eq:xiF} implies $F\vert_{\gamma}=F$ for all $\gamma\in\Gamma$. Thus $F\in\mathcal{H}^{(g)}_{1-N}$  implying $F=0$ by the Riemann-Roch theorem. Thus $\alpha$ is injective and therefore bijective.
  \eqref{eq:kernels} implies that the Bers map $\beta:  \HgN \rightarrow \COH$ is also bijective.
\end{proof}

Let $\{\Phi _{s}\}_{s=1}^{d_{N}}$  be a $\HgN$-basis with  dual basis 
$\{\Phi ^{\vee}_{r}\}_{r=1}^{d_{N}}$.
Let $F_{r}(y)=-\langle \Psi_{N}(\cdot,y), \Phi_{r} \rangle$ be a  potential for $\Phi_{r}$ with cocycle $\Xi_{r}[\gamma]:=F_{r}|_{\gamma}-F_{r}$. We then find $\chi_{N}[\gamma](x,y)$ of \eqref{eq:PsiNdefgamy} is given by
\begin{lemma}
	For all $\gamma\in\Gamma$ we have 
	\begin{align}
		\label{eq:chiPhidual}
		\chi_{N}[\gamma](x,y)=-\sum_{r=1}^{d_{N}}\Phi ^{\vee}_{r}(x)\Xi_{r}[\gamma](y).
	\end{align}
\end{lemma}
\begin{proof}
$\chi_{N}[\gamma](x,y)=\sum_{r=1}^{d_{N}}\alpha_{r}\Phi ^{\vee}_{r}(x)$ for 
$
\alpha_{r}=\langle \Psi_{N}(\cdot,\gamma y)-\Psi_{N}(\cdot,y),\Phi_{r}\rangle
=-\Xi_{r}[\gamma] (y)$.
\end{proof} 
Alternatively, for each Schottky group generator $\gamma_{a}$, 
consider the expansion
\begin{align*}
\chi_{N}[\gamma_{a}](x,y)=-\sum_{\ell\in\calL}\Theta_{N,a}^{\ell}(x)y_{a}^{\ell}\,dy^{1-N},\quad a\in\Ip,
\end{align*}
with $y_{a}:=y-w_{a}$ and some  $\Theta_{N,a}^{\ell}\in\HgN$ for $a\in\Ip$ and $\ell\in\calL$. 
Define a canonical cocycle basis $\{\Xi_{a}^{\ell}[\gamma](y)\}$ by its evaluation on each $\Gamma$ generator $\gamma_{b},\, b\in\Ip$ as follows:
\begin{align}
	\label{eq:xinorm}
	\Xi_{a}^{\ell}[\gamma_{b}](y):=\delta_{ab}y_{a}^{\ell}\, dy^{1-N},
\end{align}
Comparing to \eqref{eq:chiPhidual} we find
\begin{lemma}\label{lem:chiN}
For all $\gamma\in \Gamma$ 	we have
\begin{align}
		\label{eq:chiTheta}
		\chi_{N}[\gamma](x,y)=-\sum_{a\in\Ip}\sum_{\ell\in\calL}
		\Theta_{N,a}^{\ell}(x)\Xi_{a}^{\ell}[\gamma](y),
	\end{align}
where $\{ \Theta_{N,a}^{\ell}(x)\}$, for $a\in \Ip$ and $\ell\in\calL$,  is a $\HgN$ spanning set. 
\end{lemma}

\begin{remark}\label{rem:Psi unique}
The Bers quasi-form is not uniquely defined since \eqref{eq:PiNdef} depends on the choice of distinct limit set points $\{A_{j}\}$. Furthermore, quasi-periodicity \eqref{eq:PsiNdefgamy} implies that $\Psi_{N}(x,\gamma y)$ for any $\gamma\in\Gamma$ can be employed in \eqref{eq:FPsi} to  construct a potential $F(\gamma y)$ for each $\Phi\in\HgN$. In general,  let  $\{\Phi _{s}\}_{s=1}^{d_{N}}$ be a $\HgN$-basis with Petersson dual basis 
$\{\Phi ^{\vee}_{r}\}_{r=1}^{d_{N}}$ and corresponding potentials
$F_{r}(y)=-\langle \Psi_{N}(\cdot,y), \Phi_{r} \rangle$ with cocycles $\Xi_{r}[\gamma]:=F_{r}|_{\gamma}-F_{r}$. Then for any $d_{N}$ elements $P_{r}(y)\in\HON$ we can construct a Bers quasi-form 
	\begin{align}\label{eq:Psihat}
		\Psitilde_{N}(x,y):= \Psi_{N}(x,y)-\sum_{r=1}^{d_{N}}\Phi_{r}^{\vee}(x)P_{r}(y), 
	\end{align}
with $\Phi _{r}$ potential  $\widetilde{F}_{r}=F_{r}+P_{r}$ and cocycle $\widetilde{\Xi}_{r}=\Xi_{r}+\Xi_{P_{r}}$ for  $\Xi_{P_{r}}:= P_{r}|_{\gamma}-P_{r}$. The corresponding $\HgN$ 
spanning set of Lemma~\ref{lem:chiN} has elements
\begin{align}\label{eq:Thetatilde}
\widetilde{\Theta}_{N,a}^{\ell}(x)=\Theta_{N,a}^{\ell}(x)
+ \sum_{r=1}^{d_{N}} \Phi_{r}^{\vee}(x) p_{ra}^{\ell},
\end{align}
where $p_{ra}^{\ell}\in\C$ is determined from $\Xi_{P_{r}}[\gamma]=\sum_{\ell\in\calL}\sum_{a\in\Ip}p_{ra}^{\ell}\Xi_{a}^{\ell}[\gamma]$.
\end{remark}

\subsection{Relationship of Bers Quasiform to the Green's Function}
We consider the Green's function $G_{N}(x,y)$ for the anti-holomorphic part of the Poincar\'e metric compatible  connection \cite{Ma,McI,McIT} and its relationship to the Bers quasiform $\Psi_{N}$.  $G_{N}(x,y)$ is a unique  form of weight $(N,1-N)$  with $G_{N}(x,y)\sim \frac{1}{x-y}dx^{N}dy^{1-N}$ for $x\sim y$. $G_{N}(x,y)$ is holomorphic in $x$ for $x\neq y$ but is non-holomorphic in $y$. This contrasts with the Bers quasi-form $\Psi_{N}(x,y)$ which is holomorphic in $y$ for $y\neq x$ but is not a weight $1-N$ form in $y$.  This is very similar to the relationship between a mock modular form and its modular completion \cite{Zw,BFOR} in that a mock modular form is holomorphic but not modular whereas the modular completion is modular but not holomorphic. Following the parlance for mock modular forms, we can view $G_{N}$ as  a Schottky group completion of $\Psi_{N}$ whereas $\Psi_{N}$ is a (non-unique) holomorphic projection of $G_{N}$.

Consider a $\HgN$-basis $\{\Phi _{r}=\phi_{r}(z)dz^{N}\}$ 
with Petersson dual basis 
$\{\Phi ^{\vee}_{r}=\phi ^{\vee}_{r}(z)dz^{N}\}$ for $r=1,\ldots,d_{N}$. 
Define the projection kernel by
\begin{align}
	\label{eq:PNker}
	p_{N}(x,y):=
	\sum_{r=1}^{d_{N}}\phi _{r}^{\vee}(x) \overline{\phi _{r}(y)} \pmet^{1-N}(y).
\end{align}
For $N\ge 2$ and $g\ge 2$ we define the Green's function to be a $(N,1-N)$-form $G_{N}(x,y)=g_{N}(x,y)dx^{N}dy^{1-N}$ where the regular part defined by 
\begin{align}
	\label{eq:GN1}
	g_{N}^{R}(x,y):=g_{N}(x,y)-\frac{1}{x-y},
\end{align}
satisfies the following two conditions:
\begin{enumerate}[(I)]
	\item $g_{N}^{R}(x,y)$ is holomorphic in $x$, 
	\item $g_{N}^{R}(x,y)$ is regular but not holomorphic in $y$ with $
	\dfrac{1}{\pi}\dyb\, g_{N}^{R}(x,y)
	=  p_{N}(x,y)$. 
\end{enumerate}

\begin{remark}\leavevmode
	\begin{enumerate}[(i)]
		\item
		We may heuristically rewrite (II) as 
		\[
		\frac{1}{\pi}\dyb g_{N}(x,y)=-\delta(y-x)+p_{N}(x,y),
		\]
		for Dirac delta function $\pi\delta(y-x)=\dyb (y-x)^{-1}$.
		This  is a defining property for the Green's function in the physics literature e.g. \cite{EO,Ma}.	
		
		\item 
		We may write (II)  in a coordinate-free way (similarly to \eqref{eq:dPhiF}) where
		\begin{align}
			\label{eq:dKNPhi}
			\frac{1}{\tpi}d \left(G_{N}(x,\cdot)\Phi \right)
			=\Phi\,P_{N}(x,\cdot),
		\end{align}
		on $\D-\{x\}$ for all $\Phi \in\HgN$ with
		\[
		P_{N}(x,y):=
		\sum_{r=1}^{d_{N}}\Phi _{r}^{\vee}(x) \overline{\Phi _{r}(y)}  \PM(y)^{-N}\vol(y).
		\] 
	\end{enumerate}
\end{remark}
\begin{lemma}
	\label{lem:KNunique}
	The Green's function is unique.
\end{lemma}
\begin{proof}
	Suppose $G_{N}$ and $\widetilde{G}_{N}$ are Green's functions.  Let  
	\[
	G_{N}(x,y)-\widetilde{G}_{N}(x,y)= h_{N}(x,y)dx^{N}dy^{1-N},
	\]
	for $h_{N}(x,y)=g^{R}_{N}(x,y)-\widetilde{g}^{R}_{N}(x,y)$.
	(II) implies $ h_{N}(x,y)dy^{1-N}$ is a holomorphic form of weight $1-N<0$ in $y$ (for fixed $x$). Hence $ h_{N}=0$  by the Riemann-Roch theorem.
\end{proof} 
\begin{proposition}
	\label{prop:KNperp}
	$\langle G_{N}(\cdot,y),\Phi\rangle=0$ for all $\Phi\in\HgN$.
\end{proposition}
\begin{proof}
Let $\kappa_{s}(y)dy^{1-N} =\langle G_{N}(\cdot,y),\Phi_{s}\rangle$ for a $\HgN$-basis $\{\Phi_{s}\} $. 
	For $x,y\in \D$ we have
	\begin{align*}
		\frac{1}{\pi}\dyb\kappa_{s}(y)&=\frac{1}{\pi}\dyb\iint_{\D}g_{N}(x,y)\,
		\overline{\phi_{s}(x)}\pmet(x)^{1-N}\,d^{2}x
		\\
		&=-\overline{\phi_{s}(y)}\pmet(y)^{1-N}+\frac{1}{\pi}\iint_{\D}\dyb g_{N}^{R}(x,y)\,
		\overline{\phi_{s}(x)}\pmet(x)^{1-N}\,d^{2}x.
	\end{align*}
	Then  \eqref{eq:PNker} and condition (II)  imply
	\begin{align*}
		\frac{1}{\pi}\dyb\kappa_{s}(y)= -\overline{\phi_{s}(y)}\pmet(y)^{1-N}+\sum_{r=1}^{d_{N}}\langle \Phi _{r}^{\vee},\Phi_{s}\rangle \overline{\phi _{r}(y)} \pmet(y)^{1-N}=0.
	\end{align*}
	Thus $ \kappa_{s}(y)dy^{1-N}$ is a negative weight  holomorphic  form and hence
	$ \kappa_{s}=0$.
\end{proof}     
We now construct  the unique Green's function \cite{Ma,McIT} from the Bers quasi-form $\Psi_{N}(x,y)$.
Let $\{\Phi _{s}\} $  be a $\HgN$-basis with Petersson dual basis  $\{\Phi ^{\vee}_{r}\} $ and potentials $\{{F} _{s}(y)\}$ associated with $\Psi_{N}$.
\begin{proposition} 
	The Green's function $G_{N}(x,y)$ for $x,y\in\D$ is given by
	\begin{align}
		\label{eq:KNdefn}
		G_{N}(x,y)=\Psi_{N}(x,y)+\sum_{r=1}^{d_{N}}\Phi ^{\vee}_{r}(x){F}_{r}(y).
	\end{align}
\end{proposition}
\begin{proof}
	$ \Psi_{N}(x,y)+\sum_{r}\Phi ^{\vee}_{r}(x){F}_{r}(y)$ is an $N$-differential with respect to $x$ and using \eqref{eq:xiF} and \eqref{eq:chiPhidual}, it is a $1-N$ differential with respect to $y$. Thus condition (I) of  \eqref{eq:GN1} is verified.
	It is straightforward to confirm condition (II)  using the Bers equation \eqref{eq:dF} and  that 
	\[
	\psi_{N}^{R}(x,y)=\psi_{N}(x,y)-\frac{1}{x-y},
	\]
	is holomorphic in $x$. 
\end{proof}
The inverse map $\alpha^{-1}$  associated with the inverse Bers map $\beta^{-1}$, which exists by Proposition~\eqref{prop:BersMap}, can  be explicitly described in terms of the  Green's function as follows:
\begin{proposition}\label{prop:GNF}
	Let $\Xi$ be a 1-cocycle. Then $\Phi =\beta^{-1}(\Xi)$ has Bers potential \\ $F_{\Xi}(y)=\alpha^{-1}(\Xi)$ 
	for all $y\in\D$  given by
	\begin{align}
		\label{eq:GNF}
		F_{\Xi}(y)
		=
		\frac{1}{\tpi}\sum_{a\in\Ip}\oint_{\Con{a}}G_{N}(\cdot,y)\Xi[\gamma_{a}](\cdot).	
	\end{align}
\end{proposition}
\begin{proof} 
	Let $\{\Phi _{s}\} $  be a $\HgN$-basis with dual basis $\{\Phi ^{\vee}_{r}\}$, potentials $\{{F} _{s}(y)\}$ and cocycles  $\{\Xi_{s}\}$ associated with a Bers quasiform $\Psi_{N}$. Then $\Xi=\sum_{r=1}^{d_{N}}\xi_{r}\Xi_{r}$ where 
	\[
	\xi_{r}=\frac{1}{\tpi} \sum_{a\in\Ip}\oint_{\Con{a}}\Phi ^{\vee}_{r}\Xi[\gamma_{a}],
	\]
	using Corollary~\ref{cor:PhiCont}. Then  $F_{\Xi}(y):=\sum_{r=1}^{d_{N}}\xi_{r}F _{r}(y)$ is a potential for $\Phi(y):=\sum_{r=1}^{d_{N}}\xi_{r}\Phi _{r}(y)$ with cocycle $\Xi$. From \eqref{eq:KNdefn} we have
	\begin{align*}
		\frac{1}{\tpi}\sum_{a\in\Ip}\oint_{\Con{a}}G_{N}(\cdot,y)\Xi[\gamma_{a}](\cdot)=&
		\frac{1}{\tpi}\sum_{r=1}^{d_{N}}\xi_{r}\sum_{a\in\Ip}
		 \oint_{\Con{a}}\left(\Psi_{N}(\cdot,y)
		+\sum_{s=1}^{d_{N}}\Phi ^{\vee}_{s}(\cdot){F}_{s}(y)
		\right)\Xi_{r}[\gamma_{a}](\cdot)
		\\
		=&0+F_{\Xi}(y),
	\end{align*}
from Corollary~\ref{cor:PhiCont} and Corollary~\ref{cor:PsiXi} (shown below). 
\end{proof}
%

\section{Expansions of Meromorphic $N$-forms} \label{sec:H_N_Expansions}
We describe the expansion of meromorphic $N$-forms in terms of holomorphic $N$-forms and the Bers quasiform for genus $g\ge 2$. These expansions are natural generalizations of genus zero and one results. All of these expansions are mirrored in formal Zhu recursion formulas at each genus in the theory of vertex operator algebras \cite{TW}.  

Let $H_{N}(x)$ denote a meromorphic $N$-form for $N\ge 1$ on the Riemann surface $\Sg$ in the Schottky parameterization. Let $H_{N}(x)$ have poles at $x=y_{k}\in\D$ with $y_{k}\neq \infty$ for $k=1,\ldots ,n$
with residues
\begin{align*}
	\Res_{y_{k}}^{j}H_{N}:=\frac{1}{\tpi}\oint_{\calC_{k}}(z-y_{k})^{j}H_{N}(z)dz^{1-N}, \quad j\ge 0,\, k=1,\ldots,n,
\end{align*}
where $\calC_{k}$ is a simple Jordan curve surrounding $z=y_{k}$ but no other poles of $H_{N}(z)$.
We also define residues associated with the Schottky contours  $\calC_{a}$  below
\begin{align*}
	\Res_{w_{a}}^{\ell}H_{N}:=\frac{1}{\tpi}\oint_{\calC_{a}}z_{a}^{\ell}
	H_{N}(z)dz^{1-N}, \quad \ell\in\calL,\, a\in\I,
\end{align*}
for $z_{a}:=z-w_{a}$
for Schottky  parameter $w_{a}$. 
Recalling \eqref{eq:delconv}  we find:
\begin{proposition}\label{prop:GNexp}
	Let $H_{N}(x)$ be a meromorphic $N$-form on $\Sg$ with poles at $x=y_{k}\in\D$ with $y_{k}\neq \infty$ for $k=1,\ldots ,n$ in the Schottky parameterization. 
	\begin{enumerate}
		\item[(i)] Let $\{ \Phi_{r}(x)\}_{r=1}^{d_{N}}$ be a $\HgN$  basis  with Petersson dual basis $\{ \Phi_{r}^{\vee}(x)\}_{r=1}^{d_{N}}$ and  Bers potentials $F_{r}(y)=-\langle \Psi_{N}(\cdot,y),\Phi_{r}(\cdot)\rangle$ with $\Xi_{r}[\gamma](y)=F_{r}(\gamma y)-F_{r}(y)$. Then
		\begin{align}
			\label{eq:HNPhi}
			H_{N}(x)=&\sum_{r=1}^{d_{N}}\Phi_{r}^{\vee}(x)
			\sum_{a\in\Ip}\sum_{\ell\in\calL}\xi_{ra}^{\ell}
			\Res_{w_{a}}^{\ell}H_{N}
			+\sum_{k=1}^{n}\sum_{j\ge 0}\Psi_{N}^{(0,j)}(x,y_{k})dy_{k}^{N-1}
			\Res_{y_{k}}^{j}H_{N},
		\end{align} 
		with  $\xi^{\ell}_{ra}\in\C$ determined from $\Xi_{r}[\gamma_{a}](z)=\sum_{\ell\in\calL}\xi^{\ell}_{ra}z_{a}^{\ell}dz^{1-N}$.
		\item [(ii)]	For the $\HgN$ spanning set  $\{\Theta_{N,a}^{\ell}(x)\}$ of \eqref{eq:chiTheta} we have
		\begin{align}
			\label{eq:HNexp}
			H_{N}(x)=&\sum_{a\in\Ip}\sum_{\ell\in \calL}\Theta_{N,a}^{\ell}(x)
			\Res_{w_{a}}^{\ell}H_{N}
			+\sum_{k=1}^{n}\sum_{j\ge 0}\Psi_{N}^{(0,j)}(x,y_{k})dy_{k}^{N-1}
			\Res_{y_{k}}^{j}H_{N}.
		\end{align} 
		\item[(iii)] Let $P(x)=p(x)dx^{1-N}\in \HON$ with $\Xi_{P}[\gamma](z)=P(\gamma z)-P(z)$. Then
		\begin{align}
			\label{eq:HNP}
			-\sum_{a\in\Ip}\sum_{\ell\in\calL}p_{a}^{\ell}
			\Res_{w_{a}}^{\ell}H_{N}
			+\sum_{k=1}^{n}\sum_{\ell\in\calL}p^{(\ell)}(y_{k})
			\Res_{y_{k}}^{\ell}H_{N}=0,
		\end{align} 
		with $p_{a}^{\ell}\in\C$ determined from  $\Xi_{P}[\gamma_{a}](z)=\sum_{\ell\in\calL}p_{a}^{\ell}z_{a}^{\ell}dz^{1-N}$.
	\end{enumerate}	
\end{proposition} 
\begin{proof}
	(i) Consider a simple Jordan curve $\calC$ surrounding $x,y_{1},\ldots,y_{n}$ and the regions $\Delta_{a}$ for all $a\in\I$ so that $\oint_{\calC}\Psi_{N}(x,\cdot)H_{N}(\cdot)=0$. 
	Since $\Psi_{N}(x,z)\sim -(z-x)^{-1}\, dx^{N}dz^{1-N}$ for $x\sim z$ we find by Cauchy's theorem that
	\begin{align}
		0= -\tpi H_{N}(x)+\sum_{a\in\I}\oint_{\calC_{a}}\Psi_{N}(x,\cdot)H_{N}(\cdot)
		+\sum_{k=1}^{n}\oint_{\calC_{k}}\Psi_{N}(x,\cdot)H_{N}(\cdot).
		\label{eq:HNexp1}
	\end{align}
	Since $\calC_{-a}=-\gamma_{a} \calC_{a}$ we find using \eqref{eq:chiPhidual}
	that
	\begin{align}
		\notag
		\sum_{a\in\I}\oint_{\calC_{a}}\Psi_{N}(x,\cdot)H_{N}(\cdot)=&
		\sum_{a\in\Ip}\oint_{\calC_{a}(z)}\left(\Psi_{N}(x,z)-\Psi_{N}(x,\gamma_{a}z)\right)H_{N}(z)
		\\
		=& 
		\sum_{r=1}^{d_{N}}\Phi_{r}^{\vee}(x)
		\sum_{a\in\Ip}\oint_{\calC_{a}}\Xi_{r}[\gamma_{a}]H_{N}
		\label{eq:HNexp2}.
	\end{align}
	which determines the first term in \eqref{eq:HNPhi}. 
	Let $\Psi_{N}(x,z)=\psi_{N}(x,z)dx^{N}dz^{1-N}$. Then 
	$\psi(x,z)=\sum_{j\ge 0} \psi^{(0,j)}(x,y_{k})(z-y_{k})^{j}$ implies
	\begin{align}
		\notag
		\oint_{\calC_{k}(z)}\Psi_{N}(x,z)H_{N}(z)=& 
		\sum_{j\ge 0}\psi^{(0,j)}(x,y_{k})dx^{N}\oint_{\calC_{k}}(z-y_{k})^{j}H_{N}(z)dz^{1-N}
		\\
		\label{eq:HNexp3}
		=& \tpi\sum_{j\ge 0}\Psi^{(0,j)}(x,y_{k})dy_{k}^{N-1}
		\Res_{y_{k}}^{j}H_{N}.
	\end{align}
	Substituting \eqref{eq:HNexp2} and \eqref{eq:HNexp3} into \eqref{eq:HNexp1} completes the proof of (i).
	
	The proof of (ii) follows that of (i) where the identity \eqref{eq:HNexp2} is replaced by
	\begin{align*}
		\sum_{a\in\I}\oint_{\calC_{a}}\Psi_{N}(x,\cdot)H_{N}(\cdot)
		=
		\tpi \sum_{a\in\Ip}\sum_{\ell\in\calL}\Theta_{N,a}^{\ell}(x)
		\Res_{w_{a}}^{\ell}H_{N},
	\end{align*}
	using \eqref{eq:chiTheta}.
	(iii) follows from a similar analysis of $\oint_{\calC}PH_{N}=0$. 
\end{proof}
\begin{remark}
Following Remark~\ref{rem:Psi unique} it follows that 
\eqref{eq:HNPhi} holds for $\Psitilde_{N}(x,y)$ of \eqref{eq:Psihat} with cocycles $\widetilde{\Xi}_{r}=\Xi_{r}+\Xi_{P_{r}}$. This result also follows from \eqref{eq:HNP}.
\end{remark}

For a $\HgN$ basis $\{ \Phi_{r}(x)\}$ with cocycles $\{\Xi_{r}\}$ as in Proposition~\ref{prop:GNexp}~(i) we find:
\begin{corollary}\label{cor:PsiXi}
For all $y\in\D$ and $r\in\{1,\ldots,d_{N}\}$ we have
\[
\sum_{a\in\Ip}\oint_{\Con{a}}\Psi_{N}(\cdot,y)\Xi_{r}[\gamma_{a}](\cdot)=0.
\]
\end{corollary}
\begin{proof}
Write $\Psi_{N}(x,y)=H_{N}(x)dy^{1-N}$ for meromorphic $N$-form  $H_{N}(x)$ with a unique simple pole at $x=y\in \D$ with residue $1$. Then Proposition~\ref{prop:GNexp}~(i) and \eqref{eq:HNexp2} imply
	\begin{align*}
		\Psi_{N}(x,y)=\sum_{r=1}^{d_{N}}\Phi_{r}^{\vee}(x)
		\sum_{a\in\Ip}\frac{1}{\tpi}\oint_{\calC_{a}}\Psi_{N}(\cdot,y)\Xi_{r}[\gamma_{a}](\cdot)
		+\Psi_{N}(x,y).
	\end{align*}
	The result follows from linear independence of the $\HgN$ dual basis elements.
\end{proof}

\section{Sewing Expansion Formulas for $\Psi_{N}$ and $\Theta_{N,a}^{\ell}$ } \label{sec:Sew}
In this section we describe an expansion formula for $\Psi_{N}(x,y)$ and the holomorphic $N$-form spanning set $\{\Theta_{N,a}^{\ell}\}$ in terms of the sewing parameters $\rho_{a}$.
These expressions are natural extensions of expansions for the classical bidifferential of the second kind and holomorphic 1-forms  \cite{Fa,Y,MT,T}. 
This expression is much more suitable for applications in vertex operator algebra theory \cite{TW} than the defining Poincar\'e sum \eqref{eq:PsiNdef}. This result also determines expansion formulas for $\omega_{N}$ of \eqref{eq:Lambda} and $\Theta_{N,a}^{\ell}$ of \eqref{eq:chiTheta} and leads to an infinite determinant product formula related the determinant of a Laplacian acting on $N$-differentials on $\Sg$ \cite{McIT}.

\subsection{Sewing expansion for $\Psi_{N}(x,y)$}
For $\Pi_{N}(x,y)$ of \eqref{eq:PiNdef} we have
\begin{lemma}
	\label{lem:inteqn2} 
	For $N\ge 1$ and all $x,y\in \D$ we have 
	\begin{align}
		\Psi_{N} (x,y)=\Pi_{N}(x,y)
		-\frac{1}{\tpi}\sum_{a\in\I}\,
		\oint\limits_{\calC_{a}}\Psi_{N} (x,\cdot) \;\Pi_{N}(\cdot,y).  \label{eq:PsiNint}
	\end{align}%
\end{lemma}
\begin{proof}
	Let $\calC$ be a simple Jordan curve whose interior region contains  $\Delta_{a}$ for all $a\in \I$ and the points $x,y$. Then 
\begin{align*}
0=&	\frac{1}{\tpi}\oint\limits_{\calC}\Psi_{N} (x,\cdot) \;\Pi_{N}(\cdot,y)
\\
=&-\Pi_{N}(x,y)+\Psi_{N}(x,y)
+\frac{1}{\tpi}\sum_{a\in\I}\,
\oint\limits_{\calC_{a}} \Psi_{N} (x,\cdot) \;\Pi_{N}(\cdot,y),
\end{align*}
since $\Psi_{N}(x,z) \sim (x-z)^{-1}dx^{N}dz^{1-N}$ for $z\sim x$ and $\pi_{N}(z,y)\sim(z-y)^{-1}$ for $z\sim y$.
\end{proof} 
We note that 
\begin{align*}
	\pi_{N}(x,y)&
	=\frac{1}{x-y}+\sum_{\ell\in\calL} f_{\ell}(x)y^{\ell},
\end{align*}
for $f_{\ell}(x)=-\sum_{i\in\calL}\frac{1}{x-A_{i}}p_{i}^{(\ell)}(0)$ where $p_{i}(y)=\prod_{j\neq i}\frac{y-A_{j}}{A_{i}-A_{j}}\in \Poly_{2N-2}$ recalling \eqref{eq:delconv}.
Define
\begin{align}
	L_{b}^{n}(x):=&\frac{\rho_{b} ^{\half n}}{\tpi}
	\oint\limits_{\calC_{b}(y)}
	\Pi_{N}(x,y)y_{b}^{-n-1}\,dy^{N}
	\label{eq:Ldef}
	=\rho_{b}^{\half n} \pi_{N}^{(0,n)}(x,w_{b})dx^{N},  
	\\
	R_{a}^{m}(y):=&\frac{(-1)^{N}\rho_{a} ^{\half (m+1)}}{\tpi}
	\oint\limits_{\calC_{-a}(x)}
	\Pi_{N}(x,y)x_{-a}^{-m-1}\,dx^{1-N}
	\label{eq:Rdef}
	\\
	=&(-1)^{N}\rho_{a}^{\half (m+1)} \pi_{N}^{(m,0)}(w_{-a},y)dy^{1-N},  \notag
\end{align}
where $y_{b}:=y-w_{b}$ and $x_{-a}:=x-w_{-a}$.
We let $L(x)=(L_{b}^{n}(x))$ and 
$R(y)=(R_{a}^{m}(y))$ denote 
infinite row and column vectors, respectively, indexed by $a,b\in \I$ and $m,n\ge 0$.
We  define the doubly indexed matrix $A=(A_{ab}^{mn})$ with components 
\begin{align}
	A_{ab}^{mn} :=&
	=
	\frac{\rho_{b}^{\half n}}{\tpi}
	\oint\limits_{\calC_{b}(y)} R_{a}^{m}(y)  y_{b}^{-n-1}\,dy^{N} 
	\label{eq:Adef}
	=
	\begin{cases}(-1)^{N}\rho_{a}^{\half (m+1)}\rho_{b}^{\half n}\pi_{N}^{(m,n)}(w_{-a},w_{b}),&a\neq-b,\\ 
		(-1)^{N}\rho_{a}^{\half(m+n+1)}e^{mn}(w_{-a}),&a=-b,
	\end{cases}
\end{align}
where $e^{mn}(y):=\sum_{\ell\in\calL}\binom{\ell}{n}f_{\ell}^{(m)}(y)y^{\ell-n}$.
We note that $A_{a,-a}^{mn}=0$ for all $n> 2N-2$.
We also define a doubly indexed matrix $\Delta$ with components
\begin{align}
	\Delta_{ab}^{mn}:=\delta_{m,n+2N-1}\delta_{ab}.\label {eq:Deltadef}
\end{align}
Lastly, we define $\Ltilde(x):=L(x)\Delta$ and $\Atilde:=A\Delta $. These are independent of the $f_{\ell}(y)$ terms in $\pi_{N}(x,y)$ and are given by 
\begin{align}
	\label{eq:Ltilde}
	\Ltilde_{b}^{n}(x)&=
	\frac{\rho_{b}^{\half (n+2N-1)}}{(x-w_{b})^{n+2N}}dx^{N},
	\\
	\label{eq:Atilde}
	\Atilde_{ab}^{mn}&=
	\begin{cases}
		\displaystyle{(-1)^{m+N}\binom{m+n+2N-1}{m}
			\frac{\rho_{a}^{\half (m+1)}\rho_{b}^{\half (n+2N-1)}}
			{(w_{-a}-w_{b})^{m+n+2N}}},&a\neq-b\\ 
		0,&a=-b.
	\end{cases}
\end{align}

Let $I$ denote the infinite identity matrix and we define $(I-\Atilde)^{-1}:=I+\sum_{k\ge 1} \Atilde^{k}$.
We generalize arguments of \cite{Y}, \cite{MT} and \cite{T} to find that  
$\Psi_{N}(x,y)$ can be expressed in terms of $\Pi_{N}(x,y),\Ltilde,R$ and $\Atilde$ as follows:
\begin{proposition}
	\label{prop:Psisew} 
	For all $N\ge 1$ and $x,y\in \D$ 
	\begin{align}
		\Psi_{N} (x,y)=\Pi_{N}(x,y)+\Ltilde(x)
		(I-\Atilde )^{-1}R(y),
		\label{eq:Psisew}
	\end{align}
	where $(I-\Atilde)^{-1}$ is convergent for all $(\bm{w,\rho})\in \Cg$. 
\end{proposition} 
\begin{proof}
	Expand $\Pi_{N}(z,y)$ in $z_{-b}:=z-w_{-b}$ using \eqref{eq:Rdef} to obtain
	\[
	\Pi_{N}(z,y)=(-1)^{N}\sum_{n\ge 0}\rho_{b}^{-\half(n+1)}R_{b}^{n}(y)z_{-b}^{n}\,dz^{N}.
	\]
	Lemma~\ref{lem:inteqn2} implies
	\begin{align*}
		\Psi_{N} (x,y)
		=& \Pi_{N}(x,y)
		+(-1)^{N+1}\sum_{b\in\I}\sum_{n\ge 0}R_{b}^{n}(y)
		\frac{\rho_{b}^{-\half (n+1)}}{\tpi}
		\oint\limits_{\calC_{-b}(z)}
		\Psi_{N} (x,z)z_{-b}^{n}\, dz^{N}.
	\end{align*}
	Changing variables via $\gamma_{b}:z_{b}\rightarrow z_{-b} = \rho_{b}/z_{b}$ one finds that 
	\[
	\rho_{b}^{-\half (n+1)}z_{-b}^{n}dz_{-b}^{N}
	=(-1)^{N}\rho_{b}^{\half (n+2N-1)} z_{b}^{-n-2N}dz_{b}^{N},
	\]
	and 
	$\Psi_{N}(x,\gamma_{b}z)=\Psi_{N}(x,z)+\sum_{\ell\in\calL}\Theta_{N,b}^{\ell}(x)z_{b}^{\ell}dz^{1-N}$ from \eqref{eq:chiTheta} so that
	\begin{align*}
	\Psi_{N} (x,y)
	& =\Pi_{N}(x,y)
	+\sum_{b\in\I}\sum_{n\ge 0}R_{b}^{n}(y)
	\frac{\rho_{b}^{\half (n+2N-1)}}{\tpi}
	\oint\limits_{\calC_{b}(z)}
	\Psi_{N} (x,z)z_{b}^{-n-2N}\, dz^{N}
	\\
	&+\sum_{b\in\I}\sum_{n\ge 0}R_{b}^{n}(y)
	\frac{\rho_{b}^{\half (n+2N-1)}}{\tpi}
	\sum_{\ell\in\calL}\Theta_{N,b}^{\ell}(x)
	\oint\limits_{\calC_{b}(z)}z_{b}^{\ell-n-2N}\, dz.
	\end{align*}
	 But  $\oint_{\calC_{b}}z_{b}^{\ell-n-2N}dz=0$ for all $\ell\in\calL$ and $n\ge 0$ so we obtain
	\begin{align}
		\label{eq:PsiPU}
		\Psi_{N} (x,y)=\Pi_{N}(x,y)+\Utilde(x)\, R(y),
	\end{align}
	where $\Utilde(x)=U(x)\Delta$ for infinite row vector $U(x)$ with components 
	\begin{align*}
		U_{b}^{n}(x):=&
		\frac{\rho_{b}^{\half n}}{\tpi}
		\oint\limits_{\calC_{b}(z)}
		\Psi_{N} (x,z) z_{b}^{-n-1}\, dz^{N}.
	\end{align*}
	Integrating \eqref{eq:PsiPU} over $y\in \calC_{-b}$ and using \eqref{eq:Ldef}, \eqref{eq:Ltilde} and \eqref{eq:Atilde} we find
	\begin{align}
		\label{eq:UPR}
		\Utilde(x)=\Ltilde(x)+ \Utilde(x)\, \Atilde.
	\end{align}
	Taking $x$ moments of \eqref{eq:UPR} over $\calC_{-a}$
	we find 
	\begin{align}
		\label{eq:YRY}
		\Ytilde=(I+\Ytilde)\Atilde,
	\end{align}
	where
	$\Ytilde=Y\Delta$ for the  matrix $Y$ with components
	\begin{align}
		Y_{ab}^{mn}:=\frac{(-1)^{N}\rho_{a} ^{\half (m+1)}}{\tpi}
		\oint\limits_{\calC_{-a}(y)}
		U_{b}^{n}(x)x_{-a}^{-m-1}\,dx^{1-N}.
		\label{eq:Ydef}
	\end{align}
	Since $\Psi_{N}(x,y)$ is holomorphic for $(\bm{w,\rho})\in \Cg$ we find that $Y_{ab}^{mn}$ is convergent on $\Cg$.
	But \eqref{eq:YRY} has formal iterative solution
	\begin{align*}
		\Ytilde=\sum_{k\ge 1} \Atilde^{k}=(I-\Atilde )^{-1}-I.
	\end{align*}
	The truncation of the power series expansion of the components of 
	$ \Ytilde$ to $O\left(\rho_{a}^{M/2}\right)$, for all $a\in\I$ and for any integer $M\ge 0$,
	agrees with the equivalent truncation of $\sum_{k\ge 0}^{M} \Atilde^{k}$ since $\Atilde_{ab}^{mn}=O\left(\rho_{a}^{\half}\rho_{b}^{\half}\right)$ from \eqref{eq:Atilde}. Thus  $(I-\Atilde )^{-1}=I+\Ytilde$ is convergent for all $(\bm{w,\rho})\in \Cg$. 	
Then \eqref{eq:UPR} implies $\Utilde(x)=\Ltilde(x)(I-\Atilde )^{-1}$ and hence \eqref{eq:PsiPU} implies  \eqref{eq:Psisew}.
\end{proof}
\medskip

We now describe an expansion for the $\HgN$ spanning $N$-forms $\Theta_{N,a}^{\ell}(x)$ of \eqref{eq:chiTheta}. 
Consider the Laurent expansion
\[
\Psi_{N}(x,w_{a}+y_{a})= \ldots +T_{N,a}^{\ell}(x)y_{a}^{\ell}dy^{1-N}+\ldots 
\]
 where
\begin{align}
	\notag
	T_{N,a}^{\ell}(x):=&
	\Psi_{N}^{(0,\ell)}(x,w_{a})dy^{N-1}
	\label{eq:Tdefn}
= \rho_{a}^{-\half \ell}
	(L(x) +\Ltilde(x)(I-\Atilde)^{-1}A )_{a}^{\ell},
\end{align}
for $a\in\I$ and $\ell\in\calL$ employing \eqref{eq:Ldef} and \eqref{eq:Adef}. Using $\gamma_{a} (w_{a}+y_{a})=w_{-a}+y_{-a}$, with $y_{-a}=\rho_{a}/y_{a}$ we also find
\[
\Psi_{N}(x,\gamma_{a}y)= \ldots +
(-1)^{N+1}\rho_{a}^{N-1-\ell}T_{N,-a}^{2N-2-\ell}(x)y_{a}^{\ell}dy^{1-N}+\ldots 
\]
Comparing  $\Psi_{N}(x,y)-\Psi_{N}(x,\gamma_{a}y)$  to \eqref{eq:chiTheta} we obtain 
\begin{proposition}
For $a\in\Ip$ and $\ell\in\calL$
	\begin{align*}
		\Theta_{N,a}^{\ell}(x)=T_{N,a}^{\ell}(x)+(-1)^{N }\rho_{a}^{N-1-\ell}T_{N,-a}^{2N-2-\ell}(x).
	\end{align*}
\end{proposition}

\subsection{Expansion of $\omega_{N}(x,y)$}
 \eqref{eq:Lambda}and  \eqref{eq:Psisew} imply 
\begin{align}
	\label{eq:Lambdasew}
	\omega_{N}(x,y)=M_{N}(x,y)+\Ltilde(x)
	(I-\Atilde )^{-1}\Rtilde(y),
\end{align}
for column vector $\Rtilde(y)$ with components 
\begin{align}
	\label{eq:Rtilde}
	\Rtilde_{a}^{m}(y):=\left(R_{a}^{m}(y)\right)^{(2N-1)}dy^{2N-1}
	=(-1)^{m+N}F_{m}\frac{\rho_{a}^{\half(m+1)} }{(w_{-a}-y)^{m+2N}}dy^{N},
\end{align}
where 
\begin{align}
	\label{eq:Fm}
	F_{m}:=\binom{m+2N-1}{m}.
\end{align}
For $|x|>|y|$ we may write $M_{N}(x,y)$ as
\begin{align}
	M_{N}(x,y)=&\sum_{m\ge0}F_{m}x^{-m-2N}y^{m}dx^{N}dy^{N}=B(x) C(y),
	\label{eq:LBC}
\end{align}
where
$B(x)$, $ C (y)$ are row and column vectors with components given by  
\begin{align}
	B^{n}(x)
	:= &
	\frac{1}{\tpi F_{n}}
	\oint\limits_{\Con{0}(y)}
	y^{-n-1}M_{N}(x,y)dy^{1-N}
	=x^{-n-2N}\,dx^{N},
	\label{eq:Bdef}
	\\
	C^{m}(y)
	:= &
	\frac{1}{\tpi}
	\oint\limits_{\Con{\infty}(x)}
	x^{m+2N-1}M_{N}(x,y)dx^{1-N}
	=F_{m}y^{m}\,dy^{N},
	\label{eq:Cdef}
\end{align}
for $m,n\ge 0$ and Jordan curves $\Con{0} $  in the neighborhood of $0$ and $\Con{\infty}  $ 
in the neighborhood of $\infty$.
Note that $\Con{\infty}(x)\sim -\Con{0}(x^{-1})$ implies 
$C^{m}(y)= (-1)^{N}F_{m}B^{m}(y^{-1})$.

For  $\gamma\in \SL_{2}(\C)$ define a 
matrix  $D(\gamma)$ with components for $m,n\ge 0$ given by
\begin{align}
	\label{eq:Dmn}
	D^{mn}(\gamma)
	:= &
	\frac{1}{(\tpi)^2 F_{n}}
	\oint\limits_{\Con{\infty}(x)}
	\oint\limits_{\Con{0}(y)}
	x^{m+2N-1}y^{-n-1}M_{N}(x,\gamma y)\, dx^{1-N}dy^{1-N}
	\\
	\notag
	=&
	\frac{1}{\tpi F_{n}}
	\oint\limits_{\Con{0}}
	y^{-n-1}C^{m}(\gamma y)\,dy^{1-N}
	=
	\begin{cases}
		F_{m}F_{n}^{-1}
		\left(y_{\gamma}^{m}y_{\gamma}'^{N}\right)^{(n)}
		(0), & \gamma(0)\neq \infty,
		\\
		0, & \gamma(0)=\infty,
	\end{cases}
\end{align}
where $y_{\gamma}:=\gamma y$ and $y_{\gamma}':=\partial_{y}y_{\gamma}$. 
\begin{lemma} 
	\label{lem:Drep}
	(i) $B(x)D(\gamma)= B(\gamma^{-1} x)$ for $\gamma(0)\neq \infty$, 
	(ii) $D(\gamma) C (y)=  C (\gamma y)$ for $\gamma(0)\neq \infty$ 
	and (iii) $D(\gamma_1)D(\gamma_2)= D(\gamma_1\gamma_2)$ for $\gamma_{1}(0),\gamma_{2}(0),\gamma_{1}\gamma_{2}(0)\neq \infty$.
\end{lemma}

\begin{proof} $\SL_{2}(\C)$ invariance of $M_{N}(x,y)$ implies
	\[
	M_{N}(\gamma^{-1}x, y)
	=
	M_{N}(x,\gamma y)
	=B(x)  C (\gamma y),
	\] 
	Provided $\gamma(0)\neq \infty$ we may compute non-zero $y$  moments of this equation to obtain (i). A similar argument leads to (ii).  
	(iii) follows by considering $y$ moments of the relation $D(\gamma_{1})C(\gamma_{2}y)=C(\gamma_{1}\gamma_{2}y)$. 
\end{proof}
\begin{remark}
	Note that (iii) implies that $D$  is not a representation of $\SL_{2}(\C)$ \cite{T} e.g. for $\gamma=\bigl(\begin{smallmatrix}
		0 & 1\\ 1 & 0
	\end{smallmatrix} \bigr)$ then 
	$I=D(\gamma^2)\neq D(\gamma)^2=0$. 
\end{remark}

Define $\lambda_{a},\mu_{a}\in \SL_{2}(\C)$ for $a\in \I$ by
\begin{align}
	\lambda_{a}:= 
	\begin{pmatrix}
		\rho_{a}^{-1/2} & 0\\
		0 & 1
	\end{pmatrix}
	\begin{pmatrix}
		1 & -w_{a}\\
		0 & 1
	\end{pmatrix},
	\quad
	\mu_{a}:= 
	\begin{pmatrix}
		\rho_{a}^{1/2} & 0\\
		0 & 1
	\end{pmatrix}
	\begin{pmatrix}
		0 & 1\\
		1 & -w_{-a}
	\end{pmatrix}.
	\label{eq:lammu}
\end{align}
The sewing condition \eqref{eq:Cansew} reads $\mu_{a}z'=\lambda_{a}z$ with Schottky group generator\footnote{Note that there is an error (6.4) in \cite{T}}
\begin{align}
	\gamma_{a}=\mu_{a}^{-1}\lambda_{a}. 
	\label{eq:gammaSchot}
\end{align}
\begin{lemma}\label{lem:LRABCD}
	We may re-express $\Ltilde(x)$, $\Rtilde(y)$ and $\Atilde$ as follows:
	\begin{align}
		\Ltilde_{b}^{n}(x)&=\rho_{b}^{\half(N-1)} B^{n}(\lambda_{b} x),
		\label{eq:BP}
		\\
		\Rtilde_{a}^{m}(y)&=\rho_{a}^{\half(1-N)} C^{m}(\mu_{a} y),
		\label{eq:CQ} 
		\\
		\Atilde_{ab}^{mn}&=
		\rho_{a}^{\half(1-N)} \rho_{b}^{\half(N-1)} D^{mn}(\mu_{a}\lambda_{b}^{-1}).
		\label{eq:DR}
	\end{align}
\end{lemma}
\begin{proof}
	Since $\lambda_{b} x=\rho_{b}^{-\half}(x-w_{b})$ and 
	$\mu_{a} y=\rho_{a}^{\half}/(x-w_{-a})$ then \eqref{eq:Ltilde} and \eqref{eq:Bdef} imply \eqref{eq:BP} whereas  \eqref{eq:Rtilde} and \eqref{eq:Cdef} imply \eqref{eq:CQ}.
	Defining $y_{\gamma}:=\gamma y$ for $\gamma=\mu_{a}\lambda_{b}^{-1}$ we find
	\[
	y_{\gamma} =-\rho_{a}^{\half}\left( w_{-a}-w_{b}-\rho_{b}^{\half}y\right)^{-1},\quad 
	y_{\gamma}' =-\rho_{a}^{\half}\rho_{b}^{\half}\left( w_{-a}-w_{b}-\rho_{b}^{\half}y\right)^{-2},
	\]
	so that 
	\begin{align*}
		\left(y_{\gamma}^{m}y_{\gamma}'^{N}\right)^{(n)}
		(0)=
		(-1)^{m+N}
		\binom{m+n+2N-1}{n}
		\frac{\rho_{a}^{\half (m+N)}\rho_{b}^{\half (n+N)}}
		{\left( w_{-a}-w_{b}\right)^{m+n+2N}}.
	\end{align*}
	Finally, with $F_{m}$ of \eqref{eq:Fm}, we find a multinomial relation
	\[
	\binom{m+n+2N-1}{n}F_{m}=\binom{m+n+2N-1}{m}F_{n}
	=\binom{m+n+2N-1}{m,\,n,\,2N-1}.
	\]
	Then \eqref{eq:Atilde} and \eqref{eq:Dmn} imply \eqref{eq:DR}.
\end{proof}

\begin{remark}\label{rem:Poincare}
	Lemmas~\ref{lem:Drep} and \ref{lem:LRABCD} offer an alternative interpretation for the $\Psi_{N}$ expression \eqref{eq:Psisew} \cite{W}. Write the Poincar\'e sum for $\Psi_{N}$ as 
	\begin{align*}
		\Psi_{N}(x,y)=\Pi_{N}(x,y)+\sum_{k\ge 1}\sum_{\gamma \in \Gamma _{k}}\Pi_{N}(\gamma x,y),
	\end{align*}
	where $\Gamma_{k}$ denotes the set of Schotty group reduced word elements of length $k$ i.e. $\gamma=\gamma_{a_{1}}\ldots \gamma_{a_{k}}$ where for all adjacent generator labels $ a_{i} \neq -a_{i+1}$ for $i=1,\ldots,k-1$. In particular, we find that for each integer $k\ge 1$ 
	\begin{align}\label{eq:Gamksum}
		\sum_{\gamma \in \Gamma _{k}}\Pi_{N}(\gamma x,y)=\Ltilde(x)\Atilde^{k-1}R(y).
	\end{align}
	To obtain \eqref{eq:Gamksum} we first note from \eqref{eq:Rdef} and \eqref{eq:Ltilde} that 
	\begin{align}
		\notag
		\Ltilde(x)R(y)=&\sum_{a\in\I}\sum_{m\ge 0}
		\pi_{N}^{(m,0)}(w_{-a},y)
		\left(\frac{\rho_{a}}{x-w_{a}}\right)^{m}
		d\left(\frac{\rho_{a}}{x-w_{a}}\right)^{N}dy^{1-N}
		\\
		=& \sum_{a\in\I} \Pi_{N}(\gamma_{a}x,y)=\sum_{\gamma \in \Gamma _{1}}\Pi_{N}(\gamma x,y).
		\label{eq:PtildeQ}
	\end{align}
	We further find from Lemmas~\ref{lem:Drep} and \ref{lem:LRABCD} that for integer $k\ge 2$ 
	\begin{align*}
		\Ltilde(x)\Atilde^{k-1}R(y)= &
		\sum_{a_{1},\ldots,a_{k}\in\I}^{\prime}
		B(\lambda_{a_{k}}x)D(\mu_{a_{k}}\lambda_{a_{k-1}}^{-1})\ldots D(\mu_{a_{2}}\lambda_{a_{1}}^{-1}) \rho_{a_{1}}^{\half(N-1)}R_{a_{1}}(y)
		\notag
		\\
		= &
		\sum_{a_{1},\ldots,a_{k}\in\I}^{\prime}
		B(\lambda_{a_{1}}\gamma_{a_{2}}\ldots \gamma_{a_{k}}x) \rho_{a_{1}}^{\half(N-1)}R_{a_{1}}(y)
		\\
		= &
		\sum_{a_{1},\ldots,a_{k}\in\I}^{\prime}
		\Ltilde_{a_{1}}(\gamma_{a_{2}}\ldots \gamma_{a_{k}} x)R_{a_{1}}(y) 
		= \sum_{\gamma \in \Gamma _{k}}\Pi_{N}(\gamma x,y),
	\end{align*}
	using \eqref{eq:gammaSchot} and where the primed sum means that $ a_{i} \neq -a_{i+1}$ for $i=1,\ldots,k-1$. In particular, we note that the conditions of Lemma~\ref{lem:Drep} are satisfied since every length $k$ reduced word appears. 
\end{remark}

\subsection{ An Infinite Product Formula for $\det (I-\Atilde)$ } 
We define $\det (I-\Atilde)$  by 
\begin{align*}
	\log \det (I-\Atilde):
	=\Tr\log(I-\Atilde)
	=-\sum_{k\ge 1}\frac{1}{k}\Tr \Atilde^{k}.
\end{align*}
$\det(I-\Atilde)$ can be expressed in terms of a Selberg zeta function  infinite product formula related to the holomorphic part of determinant formula for the Laplacian acting on $N$-differentials on $\Sg$ \cite{McIT}. In the case $N=1$, this is the holomorphic part of the Montonen-Zograf formula  \cite{Mo}, \cite{Zo}, \cite{DP} whose inverse square root is the genus $g$ partition function for the Heisenberg vertex operator algebra \cite{T}.

An element $\gamma_{p}$ of the Schottky group $\Gamma$  is called primitive if $\gamma_{p}\neq \gamma^{l}$ for any $l>1$ and $\gamma\in \Gamma$. In particular, the identity element is not primitive.
Thus for each non-identity element $\gamma\in\Gamma$ we may write  $\gamma=\gamma_{p}^{l}$ for some primitive $\gamma_{p}$ and some $l\ge 1$.  We then find
\begin{theorem}
	$\det (I-\Atilde)$ is non-vanishing and holomorphic on $ \Cg$ and is given by
	\[	
	\det(I-\Atilde)=\prod_{m\ge 0}\prod_{\gamma_{p}}(1-q_{\gamma_{p}}^{m+N})
	,\]
	where $\gamma_{p}$, with multiplier $q_{\gamma_{p}}$,  is summed over representatives of the primitive conjugacy classes of $\Gamma$.
\end{theorem}
\begin{proof}
	The proof is a variation of arguments in   \cite{T}. Using \eqref{eq:DR} we find
	\begin{align*}
		\Tr \Atilde^{k}&=\sum'_{a_{1},\ldots,a_{k}\in\I}\Tr \left( D(\mu_{a_{1}}\lambda_{a_{2}}^{-1})
		D(\mu_{a_{2}}\lambda_{a_{3}}^{-1})\ldots D(\mu_{a_{k}}\lambda_{a_{1}}^{-1})\right),
	\end{align*}
	for $\lambda_{a},\mu_{a}$ of \eqref{eq:lammu} where the prime indicates that $-a_{i}\neq a_{i+1}$ for $i=1,\ldots,k-1$ and $-a_{k}\neq a_{1}$. Note that the summand trace is over the integer labels of the  $D(\gamma)$ matrices only. From Lemma~\ref{lem:Drep}~(iii), \eqref{eq:gammaSchot} and Remark~\ref{rem:Poincare} it follows that
	\begin{align*}
		\Tr \Atilde^{k}&=\sum'_{a_{1},\ldots,a_{k}\in\I}\Tr D\left(\lambda_{a_{1}}^{-1}\mu_{a_{1}}\lambda_{a_{2}}^{-1}
		\mu_{a_{2}}\ldots \lambda_{a_{k}}^{-1}\mu_{a_{k}}\right)
		\\
		&=\sum'_{a_{1},\ldots,a_{k}\in\I}
		\Tr D
		\left(\gamma_{-a_{1}}\gamma_{-a_{2}}\ldots \gamma_{-a_{k}}
		\right)
		= \sum_{\gamma\in \Gamma_{k}^{\text{CR}}}\Tr D(\gamma),
	\end{align*}
since $\gamma_{-a}=\lambda_{a}^{-1}\mu_{a}$ and 
	where $\Gamma_{k}^{\text{CR}}$ is the set of Cyclically Reduced words of length $k$ in $\Gamma$ i.e.  reduced words $\gamma_{a_{1}}\ldots \gamma_{a_{k}}$ for which $a_{1}\neq -a_{k}$.  
	For $\gamma\in \Gamma_{k}^{\text{CR}}$ we have $\gamma=\gamma_{p}^{l}$ for some primitive cyclically reduced word $\gamma_{p}$ and some $l\ge 1$ where $l|k$. 
	Every element of  $\Gamma$ is conjugate to a cyclically reduced word and any two cyclically reduced words are conjugate if and only if they are cyclic permutations of each other e.g. Prop.~9 of \cite{C}. 
	Then it follows that there are $k/l$ cyclically reduced words conjugate to $\gamma_{p}^{l}$. Therefore we find
	\begin{align*}
		\sum_{k\ge 1}\frac{1}{k}\Tr \Atilde^{k}=\sum_{\gamma_{p}}\sum_{l\ge 1}\frac{1}{l}\Tr D(\gamma_{p}^{l}),
	\end{align*}
	summing $\gamma_{p}$ over  the primitive conjugacy classes  of $\Gamma$. 
	$\gamma_{p} $ is conjugate in $\SL(2,\C)$ to $\diag(q_{\gamma_p}^{1/2},q_{\gamma_{p}}^{-1/2})$ for multiplier $q_{\gamma_{p}}$. 
	Furthermore, from \eqref{eq:Dmn} we find
	\[
	D^{mn}\left(\diag(q_{\gamma_p}^{1/2},q_{\gamma_{p}}^{-1/2})\right)=\delta_{mn}q_{\gamma_p}^{m+N}.
	\] 
	Thus we find  $\Tr \left(D^{mn}(\gamma_{p}^{l})\right)=\sum_{m\ge 0}q_{\gamma_{p}}^{l(m+N)}$  so that
	\begin{align*}
		\log\det(I-\Atilde)&=-\sum_{m\ge 0}\sum_{\gamma_{p}}\sum_{l\ge 1}\frac{1}{l}q_{\gamma_{p}}^{l(m+N)}=\sum_{m\ge 0}\sum_{\gamma_{p}}\log(1-q_{\gamma_{p}}^{m+N}).
	\end{align*}
\end{proof}

\section{Moduli Variations and Differential Operators } \label{sec:Moduli}
\subsection{Quasiconformal maps}
We consider Bers potentials, holomorphic  differentials where the $N=2$ Bers quasiform $\Psi_{2}(x,y)$  has particular geometric significance.  
The existence of a complex structure on a Riemann surface is equivalent to that of a Riemannian metric with line element
\begin{align*}
ds^2  \sim| dz +\mu(z,\zbar)\, d\zbar|^{2},
\end{align*} 
for local coordinates $z,\zbar$ where $|\mu|<1$ with  $B(z,\zbar):=\mu(z,\zbar)dz^{-1} d\zbar\in {\mathcal A} _{-1,1}$, the  Beltrami differential  e.g. \cite{GL}. The metric can be transformed to $ds^{2} \sim|dw| ^{2}$ by a quasiconformal  map $z\rightarrow w(z,\zbar)$ provided $w(z,\zbar)$ satisfies the Beltrami equation 
\begin{align}
\label{eq:beltrami}
\partial_{\zbar}w=\mu\,\partial_{z}w.
\end{align}
$B_{\Phi}(z,\zbar)=\overline{\Phi}\Lambda^{-2}\vol$ of \eqref{eq:Beldef}  for $\Phi\in\Hgtwo$ is called a harmonic Beltrami differential.
There is a 1-1 map between the infinitesimal variations of the moduli space $\Mg$ for $\Sg$ and the space of harmonic Beltrami differentials  i.e. a bijective antilinear map, known as the Ahlfors map,  between the moduli tangent space $T(\Mg)$ and $\Hgtwo$  \cite{A}.

We may explicitly realise  the Ahlfors map in the Schottky uniformization as follows. Consider a small variation in a Schottky parameter $\eta\rightarrow \eta+\varepsilon$ with corresponding  quasiconformal map given by $z\rightarrow w(z,\zbar,\varepsilon)$ where\footnote{The factor of $\frac{1}{\pi}$ is introduced to comply with our Bers potential definition \eqref{eq:dF}.}
\begin{align}\label{eq:waz}
w=z+\frac{\varepsilon}{\pi}f_{\eta}+O(\varepsilon^2),
\end{align}
for some $f_{\eta}(z,\zbar)$. 
\eqref{eq:beltrami} implies that $\mu=\varepsilon\mu_{\eta}+O(\varepsilon^2)$ where
\[
\mu_{\eta}=\frac{1}{\pi}\partial_{\zbar}f_{\eta} .
\]
Thus for a harmonic Beltrami differential,  
$F_{\eta}=f_{\eta}(z)dz^{-1}\in \calF_{2}$ is a Bers potential for $\Phi_{\eta}=\overline{\mu}_{\eta}\pmet dz^{2}\in\Hgtwo$ from \eqref{eq:dF} with $N=2$. 
The deformed Riemann surface is uniformized with a Schottky group $\Gamma_{\varepsilon}$  where for each  $\gamma\in\Gamma$ we define $\gamma_{\varepsilon}\in\Gamma_{\varepsilon}$ via the compatibility condition: $
\gamma_{\varepsilon}w(z)=w(\gamma z)$. 
But $\gamma_{\varepsilon}z=\gamma z+\varepsilon \partial_{\eta}(\gamma z)+O(\varepsilon^2)$ and using \eqref{eq:waz} we find \cite{EO,Ro,P} 
\begin{align}\label{eq:delmgamz}
\partial_{\eta}(\gamma z)=&\frac{1}{\pi}\left(f_{\eta}(\gamma z)- f_{\eta}(z)(\gamma z)'  \right)
=\frac{1}{\pi}\Xi_{\eta}[\gamma](z)d(\gamma z),
\end{align}
where $\Xi_{\eta}$ denotes  the 1-cocycle for the potential $F_{\eta}$. 

Define a $T(\Cg)$ basis $\{\partial_{a}^{\ell}\}$ for $a\in\Ip$ and $\ell=0,1,2$ given by
\begin{align}
\label{eq:delael}
\partial_{a}^{0}:=\partial_{w_{a}},\quad
\partial_{a}^{1}:=\rho_{a} \partial_{\rho_{a}},\quad
\partial_{a}^{2}:=\rho_{a} \partial_{w_{-a}}.
\end{align}
For each generator $\gamma_{b}\in\Gamma $ of \eqref{eq:gamaz} we find
\begin{align*}
\partial_{a}^{\ell}(\gamma_{b}z)=\rho_{a}(z-w_{a})^{\ell-2}\delta_{ab}=-\Xi_{a}^{\ell}[\gamma_{b}](z)d(\gamma_{b}z),
\end{align*}
where $\Xi_{a}^{\ell}$ is the canonical cocycle basis of \eqref{eq:xinorm} for $N=2$. 
Thus we find that 
\begin{align}\label{eq:delgamz}
	\partial_{a}^{\ell}(\gamma z)=-\Xi_{a}^{\ell}[\gamma ](z)d(\gamma z), \quad  \gamma\in\Gamma,
	\end{align}
giving a natural pairing of $\partial_{a}^{\ell}$ with $\Xi_{a}^{\ell}$. In conjunction with Proposition~\ref{prop:BersMap} this implies
\begin{proposition}\label{prop:TCHm1}
 $\calF_{2}\simeq \Eitwo\simeq T(\Cg)$  as vector spaces. 
\end{proposition}
By Lemma~\ref{lem:dimcobound} the  coboundary space $\Boundtwo\simeq \HOtwo$.
$\HOtwo$ is  isomorphic as a vector space to the M\"obius $\slLie_{2}(\C)$  Lie algebra where   $P=p(z)dz^{-1}\in \HOtwo$ for $p\in\Poly_{2}$ is  identified with $p(z)\partial_{z}$.
By \eqref{eq:delgamz}, the $\Boundtwo$  element of \eqref{eq:cobound} can be  written as
$\Xi_{P}[\gamma](y)= \sum_{\ell=0}^{2}\sum_{a\in\Ip}p_{a}^{\ell}\Xi_{a}^{\ell}[\gamma](y)$ for some $p_{a}^{\ell}\in\C$ and 
is paired with $\D^{P}\in T(\Cg)$ given by
\begin{align}
\D^{P}:&=\sum_{\ell=0}^{2}\sum_{a\in\Ip}p_{a}^{\ell}\partial_{a}^{\ell}
=\sum_{a\in\I}p(W_{a})\partial_{W_{a}},\label{eq:DP}
\end{align}
for the  original Schottky parameters $W_{\pm a}$. Thus $\{\D^{P}\}$ generates the $\slLie_{2}(\C)$ subalgebra of $T(\Cg)$ associated with the M\"obius action \eqref{eq:Mobwrhoa}.
%
In summary, we have the following vector space isomorphisms:
\begin{align}
\label{eq:cobound2}
\HOtwo \simeq   \Boundtwo \simeq  \slLie_{2}(\C)\subset T(\Cg).
\end{align}
Recalling that $\Schg=\Cg/\SL_{2}(\C)$ with tangent space $T(\Schg)=T(\Mg)$ we may consider the relevant quotients using Proposition~\ref{prop:TCHm1} and \eqref{eq:cobound2} to find
\begin{proposition}
	\label{prop:IsVS}
$  \Hgtwo\simeq \COHtwo \simeq T(\Mg)$ as vector spaces.
\end{proposition}
\medskip

\subsection{Differential Operators on Meromorphic  Forms}
Recall the $\Hgtwo$ spanning set $\{\Theta_{2,a}^{\ell}(x)\}$ for $a\in\Ip$ and $\ell\in\{0,1,2\}$ associated with $\Psi_{2}(x,y)$ from \eqref{eq:PsiNdefgamy} and \eqref{eq:chiTheta}. 
We define a canonical  differential operator given by \cite{GT,TW}
\begin{align}
\label{eq:nablaCg}
\delCg(x):=\sum_{\ell=0}^{2}\sum_{a\in\Ip} \Theta_{2,a}^{\ell}(x) \partial_{a}^{\ell}.
\end{align}
$\delCg(x)$ depends on $x$ and the choice of Bers quasiform $\Psi_{2}$ as described in \eqref{eq:Psihat}.
From \eqref{eq:Thetatilde} we find the differential operator associated with $\Psitilde_{2}$ is
\begin{align*}
\delCgtilde(x)=\delCg(x)+ \sum_{r=1}^{3g-3} \Phi_{r}^{\vee}(x) \D^{P_{r}},
\end{align*}
for M\"obius generator $\D^{P_{r}}$ of \eqref{eq:DP}.
Thus $\delCg(x)$ determines a unique tangent vector field, independent of the choice of Bers quasiform, on
 $T(\Schg)=T(\Mg)$ for Schottky space $\Schg=\Cg/\SL_{2}(\C)$. We denote this tangent vector field by $\delMg(x)$.   In fact, for any coordinates $\{\eta_{r}\}  $ on  moduli space $\Mg$,   the $T(\Mg)$ basis $\{\partial_{\eta_{r}}\}_{r=1}^{3g-3}$ is in 1-1 correspondence with some $\Hgtwo$-basis  $\{\Phi_{r}\} _{r=1}^{3g-3}$ via the Ahlfors map. Then for Petersson dual basis  $\{\Phi^{\vee}_{r}\} _{r=1}^{3g-3}$  we find \cite{O}
\begin{align}
\label{eq:nablaMg}
\delMg(x)=\sum_{r=1}^{3g-3}\Phi^{\vee}_{r}(x)\partial_{\eta_{r}}.
\end{align}

\medskip

$\delCg(x)$ maps differentiable functions on $\Cg$  to $\Hgtwo$. 
We generalize this to an operator which acts on meromorphic multi-variable forms as follows. 
Let $H_{\bfm}(\bfy)=h_{\bfm}(\bfy)dy_{1}^{m_{1}} \ldots dy_{n}^{m_{n}}$, for $n$ integers $\bfm:=m_{1},\cdots ,m_{n}$, denote a meromorphic form of weight $(\bfm)$ in  $\bfy:=y_{1},\cdots ,y_{n}\in \Sg$. We denote the space of  such forms by $\calM_{\bfm}^{(g)}$. 

With $d_{y}:=dy\,\partial_{y}$ for $y\in\Sg$, define the differential operators 
\begin{align}
	\nabmy{\bfm}{\bfy}(x):=&
	\nabla(x)
	+\sum_{k=1}^{n}\left(\Psi_{2}(x,y_{k})\,d_{y_{k}}+m_{k} d_{y_{k}}\Psi_{2}(x,y_{k})\right),
	\label{eq:nablaCgN}
	\\
	\D^{P}_{\bfy}:=&\D^{P}
	+\sum_{k=1}^{n}\left(P(y_{k})\,d_{y_{k}}+m_{k}d_{y_{k}}P(y_{k})\right),\quad P\in \HOtwo.
	\label{eq:DPCgN}
\end{align}
Then we find $\nabmy{\bfm}{\bfy}(x):\calM_{\bfm}^{(g)}\rightarrow \calM_{2,\bfm}^{(g)}$ as follows:
\begin{proposition}\label{prop:nablaHN}
\leavevmode
\begin{enumerate}[(i)]
\item
Let $H_{\bfm}(\bfy)\in \calM_{\bfm}^{(g)}$. Then  $\nabmy{\bfm}{\bfy}(x)\,H_{\bfm}(\bfy)$ is a meromorphic form of weight $(2,\bfm)$ in $(x,\bfy)$ with additional pole structure for  $x\sim y_{k}$ given by 
	\begin{align}
		\label{eq:NablaHNpoles}
		\nabmy{\bfm}{\bfy}(x)\,H_{\bfm}(\bfy)\sim 
		dx^{2}\left(
		\frac{m_{k}H_{\bfm}(\bfy)}{(x-y_{k})^{2}}
		+
		\frac{\del_{y_{k}}H_{\bfm}(\bfy)}{x-y_{k}}
		\right).
	\end{align}
	\item $\D^{P}_{\bfy} H_{\bfm}(\bfy)=0$ for all $P\in \HOtwo $.
\end{enumerate}
\end{proposition}
\begin{proof}
(i)  For notational simplicity we first consider the $n=1$ case of an $m$-form $H_{m}(y)=h_{m}(y)dy^{m}$ for  meromorphic $h_{m}(y)$. Then for all $\gamma\in\Gamma$
\[
h_{m}(y)=h_{m}(\ygam)\left(\ygam'\right)^{m},
\]
where $\ygam:=\gamma y$ and $\ygam':=\dy\ygam$. From \eqref{eq:nablaCgN} it is clear that $\nabmy{m}{y}(x)h_{m}(y)$
is of weight $2$ in $x$ with poles at $x=y$ as shown in \eqref{eq:NablaHNpoles}. 
We need to also show that $\nabmy{m}{y}(x)h_{m}(y)$ is of weight $m$  in $y$. 
Define $\psi_{2}(x,y):=\Psi_{2}(x,y)dy$,  a meromorphic 2-form in $x$. Then
\begin{align}
	\nabmy{m}{y}(x)\,h_{m}(y)=&\left(\delCg(x)+\psi_{2}(x,y)\partial_{y}+m\psi_{2}^{(0,1)}(x,y)\right)\left(h_{m}(\ygam)\left(\ygam'\right)^{m}\right)
\notag
\\
=& \left(\ygam'\right)^{m}
\left(\delCg(x)+\psi_{2}(x,y)\partial_{y}\right)h_{m}(\ygam)
\notag
\\
&
+mh_{m}(\ygam)\left(\ygam'\right)^{m-1}\left(\delCg(x)\ygam'+\psi_{2}(x,y)\ygam''+\psi_{2}^{(0,1)}(x,y)\ygam'\right),
\label{eq:nabhm}
\end{align}
for $\psi_{2}^{(0,1)}(x,y):=\partial_{y}\psi_{2}(x,y)$ and  $\ygam'':=\dy^2 \ygam$. Therefore 
\eqref{eq:delgamz} implies that  
\begin{align}
	\label{eq:nablayt}
	\delCg(x)\ygam
	= -\sum_{\ell=0}^{2}\sum_{a\in\Ip}\Theta_{2,a}^{\ell}(x)\Xi_{a}^{\ell}[\gamma](y)d\ygam
	=\left(\Psi_{2}(x,\ygam)-\Psi_{2}(x,y)\right)d\ygam,
\end{align} 
using \eqref{eq:chiTheta}. Thus \eqref{eq:nablayt} implies
\begin{align}
	\notag
\left(\delCg(x)+\psi_{2}(x,y)\partial_{y}\right)h_{m}(\ygam)
=&
\delCg(x)h_{m}\vert_{y}(\ygam)
+\left(\delCg(x)\ygam+\psi_{2}(x,y)\ygam'\right)\partial_{\ygam}h_{m}	(\ygam)
\\
=& \delCg(x)h_{m}\vert_{y}(\ygam)
+\psi_{2}(x,\ygam)\,\partial_{\ygam}h_{m}	(\ygam),
\label{eq:nabhmp}
\end{align}
where $\delCg(x)h_{m}\vert_{y}$ denotes differentiation with respect to the Schottky parameters for fixed $y$.
From \eqref{eq:nablayt} we also find
\begin{align}\label{eq:nabyp}
\delCg(x)(\ygam')=\partial_{y}\left(\delCg(x)\ygam\right)
=&\psi_{2}^{(0,1)}(x,\ygam)\ygam'-\psi_{2}^{(0,1)}(x,y)\ygam'-\psi_{2}(x,y)\ygam''.
\end{align}
Substituting \eqref{eq:nabhmp} and \eqref{eq:nabyp} into \eqref{eq:nabhm} we find 
\[
\nabmy{m}{y}(x)\,h_{m}(y)= \left(\ygam'\right)^{m}\nabmy{m}{\ygam}(x)h_{m}(\ygam),
\]
for all $\gamma\in\Gamma$ as required. 
The general result \eqref{eq:nablaCgN} for $n>1$ follows in  a similar way.

(ii) The second identity follows from M\"obius invariance of $H_{\bfm}(\bfy)$ cf. \secref{subsec:Schottky}. 
We again consider the $n=1$ case with $H_{m}(y)=h_{m}(y)dy^{m}$. The Schottky sewing relation \eqref{eq:Schottkysew} is invariant under M\"obius transformations
generated by $\D^{P}$ of \eqref{eq:DP}. This implies  $H_{m}(y,W_{a},q_{a})=H_{m}( y+\varepsilon p(y),W_{a}+\varepsilon p(W_{a}),q_{a})+O(\varepsilon^2)$ for $p\in\Poly_{2}$ (where we make the dependence of $H_{m}$ on the original  Schottky parameters explicit). Thus it follows that
\begin{align*}
&h_{m}(y,W_{a},q_{a})dy^{m}=  h_{m}(y+\varepsilon p(y),W_{a}+\varepsilon p(W_{a})d(y+\varepsilon p(y))^{m}+O(\varepsilon^2)
\\
&=  H_{m}(y,W_{a},q_{a}) +\varepsilon\left( \sum_{a\in\I}p(W_{a})\partial_{W_{a}}
+p(y)\partial_{y}+mp^{(1)}(y)
\right)H_{m}(y,W_{a},q_{a})+O(\varepsilon^2).
\end{align*}
This implies (ii) for $P(z)=p(z)dz^{-1}$. The result for $n\ge 2$ follows similarly.
\end{proof}
\begin{remark}\label{rem:nablarem}
	\leavevmode
\begin{enumerate}
\item Proposition \ref{prop:nablaHN}~(ii) implies  $\nabmy{\bfm}{\bfy}(x)$ determines a unique differential operator for meromorphic forms on $\Sg$ independent of the choice of the Bers quasiform $\Psi_{2}(x,y)$. This generalizes the replacement of $\nabla(x)$ by $\delMg(x)$ of \eqref{eq:nablaMg}.
\item For meromorphic forms $H_{\bfm}(\bfy)$, $G_{\bfm'}(\bfy)$,  of respective weights  $(\bfm)$, $(\bfm')$, then $H_{\bfm}G_{\bfm'}$ is a meromorphic form of  weight  $(\bfm)+(\bfm')$ and the Leibniz rule holds
\begin{align*}
\nabmy{\bfm+\bfm'}{\bfy}(x)\left(H_{\bfm}G_{\bfm'}\right)
=\left(\nabmy{\bfm}{\bfy}(x)H_{\bfm}\right)G_{\bfm'}
+H_{\bfm}\left(\nabmy{\bfm'}{\bfy}(x)G_{\bfm'}\right).
\end{align*}
\item If $H_{\bfm}(\bfy)$ is independent of $y_k$ then $m_{k}=0$ and  $\nabmy{\bfm}{\bfy}(x)\,H_{\bfm}=\nabmy{\widehat{\bfm}}{\widehat{\bfy}}(x)\,H_{\bfm}$ where $\widehat{\bfm},\widehat{\bfy}$ are  $n-1$ labels formed by deleting the $k$-th entries of ${\bfm},{\bfy}$.
\end{enumerate}
\end{remark}
\medskip
The  $\nabmy{\bfm}{\bfy}(x)$ differential operators enter into  differential equations for the following classical differentials e.g. \cite{Mu,Fa}
\begin{itemize}
	\item  the symmetric meromorphic bidifferential $\omega(x,y)$ (i.e. $\omega_{1}(x,y)$ of  \eqref{eq:Lambda}) of the second kind	 of weight $(1,1)$ in $x,y$,
	\item the meromorphic quasi-differential $\omega_{y-z}(x):=\int^{y}_{z}\omega(x,\cdot)$ of weight $(1,0,0)$ in $(x,y,z)$,
	\item the projective connection $s(x):=6\,dx^{2}\lim_{x \rightarrow y}\left(\omega(x,y)dx^{-1}dy^{-1}-\frac{1}{(x-y)^{2}}\right)$ of weight $2$, 
	\item the holomorphic differentials  $\nu_{a}(x):=\oint_{\beta_{a}}\omega(x,\cdot)$ for homology cycle $\beta_{a}$ with $a\in\Ip$, of weight 1,
	\item the Jacobian variety coordinates $\int_{y}^{x}\nu_{a}$, a quasi-differential of weight $(0,0)$ in $(x,y)$,
	\item the period matrix $\Omega_{ab}:=\frac{1}{\tpi}\oint_{\beta_{a}}\nu_{b}$ for $a,b\in\Ip $, 
	\item the prime form $E(x,y)$ (with a multiplier system) of weight $\left(-\half,-\half\right)$ in $(x,y)$. 
\end{itemize}    
\begin{proposition}\label{prop:delsForm}
$\omega(x,y)$, $s(x)$, $\omega_{y-z}(x)$,  $\nu_{a}(x)$, $\int_{y}^{x}\nu_{a}$,  $\Omega_{ab}$ and $E(x,y)$ for $a,b\in\Ip$ satisfy the following differential equations:
	\begin{align}
		&\nabmy{1,1}{y,z}(x)\,\omega(y,z) = \omega(x,y)\omega(x,z),
		\label{eq:Nabla_omega}
		\\[4pt]
		&\nabmy{2}{y}(x)\,s(y)
		 = 6 \left(\omega(x,y)^{2}-\omega_{2}(x,y)\right), 
		 \label{eq:Nabla_s} 
		\\[4pt]
		&\nabmy{1,0,0}{w,y,z}(x)\,\omega_{y-z}(w)
		 = \omega_{y-z}(x)\omega(x,w),
		\label{eq:Nabla_omz1z2} 
		\\[4pt]
		&\nabmy{1}{y}(x)\,\nu_{a}(y) = \omega(x,y)\nu_{a}(x),
		\label{eq:Nabla_nu}
		\\[4pt]
		&\nabmy{0,0}{y,z}(x) \int_{z}^{y}\nu_{a} = \omega_{y-z}(x)\nu_{a}(x),
		\label{eq:Nabla_Ja}
		\\[4pt]
		&\tpi \nabla(x)\,\Omega_{ab} = \nu_{a}(x)\nu_{b}(x)
		\label{eq:Rauch},
		\\[4pt]
		&\nabmy{0,0}{y,z}(x)\,\log E(y,z) = -\half\omega_{y-z}^{2}(x),
		\label{eq:Nabla_E}
	\end{align}
where $\omega_{2}(x,y)$ is the symmetric differential $(2,2)$ form of  \eqref{eq:Lambda}.
\end{proposition} 
\begin{remark}
Some similar differential equations previously appeared in \cite{O}.  However, neither the Bers function $\Psi_{2}$ is defined nor is  Proposition~\ref{prop:nablaHN} established there. \eqref{eq:Rauch} is equivalent to the classical Rauch formula \cite{Ra}. 
Proposition~\ref{prop:delsForm} is derived in an alternative way using  vertex operator algebras in \cite{TW} (see \cite{GT,W} also) by considering genus $g$ Virasoro Ward identities for the generalized Heisenberg vertex algebra \cite{T}.
\end{remark} 
\begin{proof}[Proof of Proposition~\ref{prop:delsForm}]
Let $F(x,y,z):=\nabmy{1,1}{y,z}(x)\,\omega(y,z)-\omega(x,y)\omega(x,z)$. $F$ is convergent at  $x=y$ and $x= z$ from \eqref{eq:NablaHNpoles} but also for $y=z$ by inspection. Thus $F$ is a holomorphic form in $x,y,z$ of weight $(2,1,1)$. Choose  a standard $\Hg_{1}$ basis $\{\nu_{a}\} $ with normalization 
\[
\oint_{\alpha_{a}}\nu_{b}(\cdot)=\oint_{\Con{-a}}\nu_{b}(\cdot)=\delta_{ab},\quad a,b,\in\Ip.
\]
Hence $F=\sum_{a,b\in\Ip}\Phi_{ab}(x) \nu_{a}(y)\nu_{b}(z)$ for 
$
\Phi_{ab}(x):=\oint_{\Con{-a}(y)}\oint_{\Con{-b}(z)}F(x,y,z)\in\Hgtwo
$.
But
\begin{align*}
\oint_{\Con{-a}(y_{1})}\oint_{\Con{-b}(z)}\nabmy{1,1}{y,z}(x)\,\omega(y,z)=&\oint_{\Con{-a}(y)}\oint_{\Con{-b}(z)}\Big( \delCg(x)\omega(y,z)+d_{y}\left(\Psi_{2}(x,y)\omega(y,z)\right)
\\
&\quad  +d_{z}\left(\Psi_{2}(x,z)\omega(y,z)\right)\Big)
=0,
\end{align*}
using $\oint_{\Con{-a}}\omega(x,\cdot)=0$ and the periodicity of $\Psi_{2}(x,y)\omega(y,z)$ in $y$ around $\Con{-a}$. Therefore  $\Phi_{ab}(x)=0$ and \eqref{eq:Nabla_omega} follows.

\eqref{eq:Nabla_s}-\eqref{eq:Nabla_E} are derived from various limits and integrals of \eqref{eq:Nabla_omega}. Thus  \eqref{eq:Nabla_s} follows from the $z\rightarrow y$ limit of \eqref{eq:Nabla_omega}.
To prove  
\eqref{eq:Nabla_omz1z2} consider
\begin{align*}
\nabmy{1,0,0}{w,y,z}(x)\,\omega_{y-z}(w)&=
\nabmy{1}{w}(x)\int_{z}^{y}\omega(w,\cdot)+\Psi_{2}(x,y)\omega(w,y)-\Psi_{2}(x,z)\omega(w,z)
\\
&=\int_{u=z}^{y}
\left(\nabmy{1}{w}(x)\omega(w,u)
+d_{u}\left(\Psi_{2}(x,u)\omega(w,u)\right)\right) 
\\
&=\int_{u=z}^{y}\nabmy{1,1}{w,u}(x)\omega(w,u)
=\omega(x,w)\int_{z}^{y}\omega(x,\cdot)=\omega_{y-z}(x)\omega(x,w).
\end{align*}

In order to prove \eqref{eq:Nabla_nu}, we recall that
$\beta_{a}$ denotes a path from $z\in  \Con{a}$ to $\zgam:=\gamma_{a}z\in  \Con{-a}$ for $a\in \Ip$. Let $g(y,w):=\omega(y,w)dw^{-1}$ denote a 1-form in $y$. Then using  \eqref{eq:nablayt} we find
\begin{align*}
\nabmy{1}{y}(x)\,\nu_{a}(y)=& \nabla(x)\int_{w=z}^{\zgam}g(y,w)dw+\int_{w=z}^{\zgam}d_{y}\left(\Psi_{2}(x,y)\omega(y,w)\right)
\\
=&
\int_{w=z}^{\zgam}\nabmy{1}{y}(x)\omega(y,w)
+(\nabla(x)\zgam)\,g(y,\zgam)
\\
=& \int_{w=z}^{\zgam}\nabmy{1}{y}(x)\omega(y,w)
+\left(\Psi_{2}(x,\zgam)-\Psi_{2}(x,z)\right)\omega(y,\zgam)
\\
=& \int_{w=z}^{\zgam}\left(\nabmy{1}{y}(x)\omega(y,w)
+d_{w}\left[\Psi_{2}(x,w)\omega(y,w)\right]\right)
\\
=&\int_{\beta_{a}(w)}\nabmy{1,1}{y,w}(x)\omega(y,w)=\omega(x,y)\nu_{a}(x).
\end{align*}
 \eqref{eq:Nabla_Ja} and Rauch's formula \eqref{eq:Rauch} follow in a  similar way.

Finally we note that 
 $
 E(y,z)=-E(z,y)\sim \left(y-z \right) dy^{-1/2}dz^{-1/2}$ for $y\sim z
 $ 
 and  $\omega_{y-z}(x)=d_{x}\log(E(x,y)/E(x,z))$  (e.g. \cite{Mu,Fa}) so that
\begin{align*}
 \int_{p}^{q}\omega_{y-z}&=\log\left(\frac{E(y,q)E(p,z)}{E(q,z)E(y,p)}\right)
\sim 2 \log E(y,z)-\log(q-z)(y-p),
\end{align*}
for $y\sim p$ and $z\sim q$. We find using \eqref{eq:Nabla_omz1z2} that for $y\sim p$ and $z\sim q$
\begin{align*}
2\,\nabmy{0,0}{y,z}(x)\log E(y,z)\sim& \int_{w=p}^{q}\nabmy{0,0}{y,z}(x)\omega_{y-z}(w)
-\Psi_{2}(x,y)\frac{dy}{p-y} 
-\Psi_{2}(x,z)\frac{dz}{q-z} 
\\
=& \int_{w=p}^{q}\nabmy{1,0,0}{w,y,z}(x)\omega_{y-z}(w)
-\int_{w=p}^{q}d_{w}\left(  \Psi_{2}(x,w)\omega_{y-z}(w)\right)
\\
&
-\Psi_{2}(x,y)\frac{dy}{p-y} 
- \Psi_{2}(x,z)\frac{dz}{q-z}  
\\
=&\omega_{y-z}(x)\omega_{q-p}(x)
+\Psi_{2}(x,p)\omega_{y-z}(p)
-\Psi_{2}(x,y)\frac{dy}{p-y} 
\\
&-\Psi_{2}(x,q)\omega_{y-z}(q)- \Psi_{2}(x,z)\frac{dz}{q-z} \sim -\omega_{y-z}^{2}(x),
\end{align*}
since $\omega_{y-z}(p)\sim \frac{dp}{p-y}$ and $\omega_{y-z}(q)\sim \frac{dq}{q-z}$  for  $y\sim p,z\sim q$. 
Thus the proposition holds.
\end{proof}
\medskip

It is natural to consider the composition  $\nabmy{2,\bfm}{x,\bfz}(y)\nabmy{\bfm}{\bfz}(x)$ of two differential operators. Such compositions arise from conformal Ward identities for genus $g$ vertex operator algebras where the locality of vertex operators implies that the composition is symmetric in $x$ and $y$ \cite{TW}. Thus we may consider 
\begin{align*}
\tpi \nabmy{2}{x}(y) \nabla(x)\,\Omega_{ab} &= \nabmy{2}{x}(y)(\nu_{a}(x)\nu_{b}(x))
=\omega(x,y)(\nu_{a}(x)\nu_{b}(y)+\nu_{a}(y)\nu_{b}(x)),
\end{align*}
using the Leibniz rule, \eqref{eq:Nabla_nu} and \eqref{eq:Rauch}. This expression is symmetric in $x,y$. But we may choose $3g-3$ independent values of $\Omega_{ab}$ as local coordinates in the neighborhood of any point in $\Mg$ so it follows from \eqref{eq:nablaMg} that
\begin{align}
	\label{eq:NablaCom1}
	\nabmy{2}{x}(y) \nabla(x)=\nabmy{2}{y}(x) \nabla(y).
\end{align}
This result can be generalized as follows
\begin{proposition}[Commutativity]
	\label{prop:NabCom}
For the differential operator \eqref{eq:nablaCgN} we have
\begin{align}
	\label{eq:NablaCom2}
  \nabmy{2,\bfm}{x,\bfz}(y)\nabmy{\bfm}{\bfz}(x)
  =
  \nabmy{2,\bfm}{y,\bfz}(x)\nabmy{\bfm}{\bfz}(y).
\end{align}
\end{proposition}
\begin{proof}
We begin by showing that\footnote{Note that although $	\Psi_{2} (x,z)$ and $d_{z}\Psi_{2} (x,z)$ are not meromorphic forms in $z$, the differential operators $\nabmy{2,-1}{x,z}(y)$ and $\nabmy{2,0}{x,z}(y)$ are well-defined.}
\begin{align}
	\label{eq:nabPsi2}
	&  
A(x,y,z):=\nabmy{2,-1}{x,z}(y)	\Psi_{2} (x,z)+\Psi_{2} (x,z) \,d_{z} \Psi_{2} (y,z),
	\\
	\label{eq:nabPsi2p}
	& 
	B(x,y,z):=\nabmy{2,0}{x,z}(y)	\,d_{z}\Psi_{2} (x,z),
\end{align}
are both symmetric in $x,y$. 
From \eqref{eq:Nabla_omz1z2} and \eqref{eq:Nabla_Ja} we note that
\begin{align}
	\label{eq:NabNab}
	\nabmy{2,0,0}{x,z,w}(y)\nabmy{0,0}{z,w}(x)\int_{w}^{z}\nu_{a}
	&=\omega(x,y)\left(\omega_{z-w}(x)\nu_{a}(y)+\omega_{z-w}(y)\nu_{a}(x)\right),
\end{align}
which is symmetric in $x,y$ and antisymmetric in $z,w$.
It follows using Remark~\ref{rem:nablarem}~(2) and (3) that 
\begin{align*}
\nabmy{2,0,0}{x,z,w}(y)\nabmy{0,0}{z,w}(x)\int_{w}^{z}\nu_{a}
=&\nabmy{2,0,0}{x,z,w}(y)	\left(\nabla(x)\int_{w}^{z}\nu_{a} +\Psi_2(x,z)\nu_{a}(z)-\Psi_2(x,w)\nu_{a}(w)\right)
\\	=&
	\left(\nabmy{2}{x}(y)+\Psi_{2}(y,z)d_{z}+\Psi_{2}(y,w)d_{w} \right)\nabla(x)\int_{w}^{z}\nu_{a}
	\\
	& +\nu_{a}(z)\nabmy{2,-1}{x,z}(y)\Psi_2(x,z)+
\Psi_2(x,z)	\nabmy{1}{z}(y) \nu_{a}(z)
	\\
& -\nu_{a}(w)\nabmy{2,-1}{x,w}(y)\Psi_2(x,w)-
\Psi_2(x,w)	\nabmy{1}{w}(y) \nu_{a}(w)
\\
=& \nabmy{2}{x}(y)\nabla(x)\int_{w}^{z}\nu_{a}
+\Psi_{2}(y,z)\left(\omega(x,z)\nu_{a}(x)-d_{z}(\Psi_{2}(x,z)\nu_{a}(z))\right)
\\
&
-\Psi_{2}(y,w)\left(\omega(x,w)\nu_{a}(x)-d_{w}(\Psi_{2}(x,w)\nu_{a}(w))\right)
\\
 &+\nu_{a}(z)\nabmy{2,-1}{x,z}(y)\Psi_2(x,z)+
\Psi_2(x,z)	\omega(y,z) \nu_{a}(y)
\\
&-\nu_{a}(w)\nabmy{2,-1}{x,w}(y)\Psi_2(x,w)-
\Psi_2(x,w)	\omega(y,w) \nu_{a}(y),
\end{align*}
using $\nabla(x)d_{z}\int_{w}^{z}\nu_{a}=\omega(x,z)\nu_{a}(x)-d_{z}(\Psi_{2}(x,z)\nu_{a}(z))$ etc. Thus we find 
\begin{align*}
\nabmy{2,0,0}{x,z,w}(y)\nabmy{0,0}{z,w}(x)\int_{w}^{z}\nu_{a}
= A(x,y,z)\nu_a(z) - A(x,y,w)\nu_{a}(w) +\ldots,
\end{align*}
where the ellipsis refers to a sum of terms explicitly symmetric in $x,y$ and antisymmetric in $z,w$. 
It follows  from \eqref{eq:NabNab} that
\begin{align*}
	(A(x,y,z)-A(y,x,z))\nu_a(z) =(A(x,y,w)-A(y,x,w))\nu_{a}(w)=F(x,y),
\end{align*}
for some meromorphic $(2,2)$ quasiform $F(x,y)$ antisymmetric in $x,y$ and independent of $z,w$. It is straightforward to check that $A(x,y,z)-A(y,x,z)$ is convergent at $z=x,y$ and $x=y$ and is therefore holomorphic. But $\nu_{a}(z)^{-1}$ is meromorphic by the Riemann-Roch theorem so that $\nu_{a}(z)^{-1}F(x,y)=0$. Hence $A(x,y,z)$ is symmetric in $x,y$.

To show $B(x,y,z)$ is symmetric in $x,y$ we consider 
\begin{align*}
	d_{z}A(x,y,z)=& \nabla(y)\, d_{z}\Psi_{2}(x,z) +\Psi_{2}(y,x)\,d_{x}d_{z} \Psi_{2}(x,z)+2 d_{x}\Psi_{2}(y,x)\,d_{z}\Psi_{2}(x,z)
	\\
	&+\Psi_{2}(y,z) \, d_{z}^{2}\Psi_{2}(x,z)+d_{z}\Psi_{2}(y,z) \, d_{z}\Psi_{2}(x,z) 
	\\
	=&  B(x,y,z)+d_{z}\Psi_{2}(y,z) \, d_{z}\Psi_{2}(x,z).
\end{align*}
Therefore $B(x,y,z)$ is symmetric in $x,y$ because $A(x,y,z)$ is.

We now proceed to the proof of \eqref{eq:NablaCom2}. Using Remark~\ref{rem:nablarem}~(2) and (3) we find
\begin{align*}
&\nabmy{2,\bfm}{x,\bfz}(y)\nabmy{\bfm}{\bfz}(x)
=
\nabmy{2,\bfm}{x,\bfz}(y)\left( \nabla(x)
+\sum_{i=1}^{n}\left( \Psi_{2}(x,z_{i})\,d_{z_{i}}+m_{i}d_{z_{i}}\Psi_{2}(x,z_{i})\right)\right)
\\
&=\nabmy{2}{x}(y) \nabla(x)
+\sum_{j=1}^{n}
\Psi_{2}(y,z_{j})\nabla(x)d_{z_{j}} +\sum_{j=1}^{n}m_{j}d_{z_{j}}\Psi_{2}(y,z_{j}).\nabla(x)
\\
&\quad +\sum_{i=1}^{n}\nabmy{2,-1}{x,z_{i}}(y)\Psi_{2}(x,z_{i})d_{z_{i}}
+\sum_{i=1}^{n} m_{i}d_{z_{i}}\Psi_{2}(x,z_{i}).\nabmy{2,-1}{x,z_{i}}(y)
\\
& \quad 
+\sum_{i=1}^{n}\Psi_{2}(x,z_{i})
\nabla(y)d_{z_{i}}
+\sum_{i=1}^{n}\Psi_{2}(x,z_{i})\sum_{j=1}^{n}\left(
\Psi_{2}(y,z_{j})d_{z_{j}}
+(m_{j}+\delta_{ij})d_{z_{j}}\Psi_{2}(y,z_{j})
\right)d_{z_{i}}
\\
&
\quad+\sum_{i=1}^{n} m_{i}d_{z_{i}}\Psi_{2}(x,z_{i})
\nabla(y) 
+\sum_{i=1}^{n} m_{i}d_{z_{i}}\Psi_{2}(x,z_{i})\sum_{j=1}^{n}\left(\Psi_{2}(y,z_{j})d_{z_{j}}+ m_{j}d_{z_{j}}\Psi_{2}(y,z_{j})\right),
\end{align*}
noting that for any meromorphic form $H_{\bfm}(\bfz)$ of weight $\bfm$, then $d_{z_{i}}\Hm$ is of weight $m_{j}+\delta_{ij}$ in $z_{j}$. Rearranging the previous expression we find
\begin{align*}
	\nabmy{2,\bfm}{x,\bfz}(y)\nabmy{\bfm}{\bfz}(x) 
=&\nabmy{2}{x}(y) \nabla(x) 
   +\sum_{i=1}^{n}\left(A(x,y,z_{i})d_{z_{i}}+m_{i}B(x,y,z_{i})\right) 	
   \\
   & 
   +\sum_{i=1}^{n}\left(	\Psi_{2}(y,z_{i})\nabla(x)+\Psi_{2}(x,z_{i})\nabla(y) \right)d_{z_{i}}
	\\
	& +\sum_{i=1}^{n}m_{j}
	\left(d_{z_{i}}\Psi_{2}(y,z_{i})\nabla(x)+d_{z_{i}}\Psi_{2}(x,z_{i})\nabla(y)\right) 
	\\
	&+\sum_{i=1}^{n}\sum_{j=1}^{n}\left(m_{i}m_{j}d_{z_{i}}\Psi_{2}(x,z_{i}).
	d_{z_{j}}\Psi_{2}(y,z_{j})+\Psi_{2}(x,z_{i})
	\Psi_{2}(y,z_{j})d_{z_{i}}d_{z_{j}}\right) 
	\\
	&+\sum_{i=1}^{n}\sum_{j=1}^{n}
	m_{j}\left(\Psi_{2}(x,z_{i})d_{z_{j}}\Psi_{2}(y,z_{j})
	+\Psi_{2}(y,z_{i})d_{z_{j}}\Psi_{2}(x,z_{j})
	\right)d_{z_{i}},
\end{align*}
which is symmetric in $x,y$. Hence the result holds.
\end{proof}

\end{document}